\documentclass[final]{siamart190516}
 \usepackage{marginfix}
 \usepackage{amsfonts}

 \usepackage{amsmath}
 \usepackage{amssymb}
 \usepackage{enumerate}
 \usepackage{graphicx}  
 \usepackage{color}
 \usepackage{longtable}
 \usepackage{multirow}
 \usepackage{float}
 \usepackage{booktabs}
 \usepackage{makecell}
 \usepackage{supertabular}
 \usepackage{algpseudocode}
 \usepackage{algorithmicx,algorithm}
 \usepackage{caption}
 \usepackage{subfigure}
 \usepackage{hyperref}
 
 \numberwithin{equation}{section}

 \newtheorem{remark}{Remark}[section]
 \newtheorem{example}{Example}

 \allowdisplaybreaks[4]

\begin{document}

\title{An efficient implementation algorithm for quasi-Monte Carlo approximations of high-dimensional integrals}

\headers{AN EffICEINT ALGRITHM FOR HIGH-DIMENSIONAL INTEGRATION}{HUICONG ZHONG AND XIAOBING FENG} 

\author{
	Huicong Zhong\thanks{School of Mathematics and Statistics, Northwestern Polytechnical University, Xi’an, Shaanxi 710129, China (\email{huicongzhong@mail.nwpu.edu.cn}).
     }
	\and 	
	Xiaobing Feng\thanks{Department of Mathematics, The University of Tennessee,  Knoxville, TN, 37996 (\email{xfeng@utk.edu}).} 
}

\maketitle

\begin{abstract}
In this paper, we develop and test a fast numerical algorithm, called MDI-LR,  for efficient implementation of quasi-Monte Carlo lattice rules for computing $d$-dimensional integrals of a given function.  It is based on the idea of converting and improving the underlying lattice rule into a tensor product rule by an affine transformation, and adopting the multilevel dimension iteration approach which computes  the function evaluations (at the integration points) in the tensor product  multi-summation in cluster and iterates along each (transformed) coordinate direction so that a lot of computations can be reused. 
The proposed algorithm also eliminates the need for storing integration points and computing function values independently at each point.
Extensive numerical experiments are presented to gauge the performance of the algorithm MDI-LR and to compare it with standard implementation of quasi-Monte Carlo lattice rules. It is also showed numerically that the algorithm MDI-LR can achieve a computational complexity of order $O(N^2d^3)$ or better, where $N$ represents the number of points in each (transformed) coordinate direction and $d$ standard for the dimension.  Thus, the algorithm MDI-LR effectively 
overcomes the curse of dimensionality and revitalizes QMC lattice rules for  high-dimensional integration. 
\end{abstract}



\section{Introduction}\label{sec-1}
Numerical integration is an essential tool and building block in many scientific and engineering fields which requires to evaluate or estimate integrals 
of given (explicitly or implicitly) functions, which becomes very challenging in high dimensions due to the so-called {\em curse of the conditionality} (CoD). 
They are seen in evaluating quantities of stochastic interests, solving high-dimensional partial differential equations, or computing value functions of 
an option of a basket of securities. 
The goal of this paper is to develop and test an efficient algorithm based on quasi-Monte Carlo methods for evaluating the $d$-dimensional  integral
\begin{equation}\label{eq1.1}
	I_{d}(f):=\int_\Omega  f(\mathbf{x})d\mathbf{x} 
\end{equation}
for a given function $f: \Omega:=[0,1]^d\to \mathbf{R}$ and $d>>1$ .

Classical numerical integration methods, such as  tensor product and sparse grid methods \cite{BG14,GG} as well as Monte Carlo(MC)  methods  \cite{Caflisch98,Ogata89} require the evaluation of a function at a set of integration points. The computational complexity of the first two types of  methods grows exponentially with the dimension $d$ in the problem (i.e.., the CoD), which limits their practical usage. 
Monte Carlo (MC) methods  are often the default methods for high dimensional integration problems due to their ability of handling complicated functions and mitigating the CoD. We recall that the MC method approximates the  integral by randomly sampling points within the integration domain and averaging their function values. The classical MC method has the form
\begin{equation}\label{eq1.2}
	Q_{n,d}(f)=\frac{1}{n}\sum_{i=0}^{n-1}f(\mathbf{x}_i),  
\end{equation}
where $\{ \mathbf{x}_i\}_{i=0}^{n-1}$ denotes independent and uniformly distributed random samples in the integration domain $\Omega$. The expected error for the MC  method is proportional to  $\frac{\sigma(f)}{\sqrt{n}}$, where $\sigma(f)^2$ stands for the variance of $f$.  If $f$ is square-integrable then the expected error in \eqref{eq1.2} has the order $O(n^{-\frac12})$ (note that the convergence rate is independent of the dimension $d$). Evidently, the MC method is simple and easy to implement, making them a popular choice for many applications. However, the MC method is slow to converge, especially for high-dimensional problems, and the accuracy of the approximation depends on the number of random samples. One way to improve the convergence rate of the Monte Carlo method is to use quasi-Monte Carlo methods.

Quasi-Monte Carlo (QMC) methods \cite{DKS13,KSS11} employ integration techniques that use point sets with better distribution properties than random sampling.  Similar to the MC method, the QMC method  also has the general form  \eqref{eq1.2},  but unlike the MC method,   the integration points   $\{ \mathbf{x}_i\}_{i=0}^{n-1}\in \Omega$ are chosen deterministically and methodically.  The deterministic nature of the QMC method  could lead to guaranteed error bounds and that the convergence rate could be faster than the $O(n^{-\frac12})$  order of the MC  method  for sufficiently smooth functions. 
QMC error bounds are typically given in the form of Koksma-Hlawka-type inequalities as follows:
\begin{equation}
	|I_d(f)-Q_{n,d}(f)|\leq D(\mathbf{x}_0,\mathbf{x}_1,\cdots,\mathbf{x}_{n-1})V(f),
\end{equation}
where $D(\mathbf{x}_0,\mathbf{x}_1,\cdots,\mathbf{x}_{n-1})$ is a (positive) discrepancy function which measures the non-uniformity of the
point set $\{ \mathbf{x}_i\}_{i=0}^{n-1}$ and $V (f)$ is a (positive) functional which measures the variability of $f$.  
Error bounds of this type separate the dependence on the cubature points from the dependence on the integrand. The QMC point sets with discrepancy
of order $O(n^{-1}(\log n)^d)$ or better are collectively known as low-discrepancy point sets \cite{HNie}.
 
 One of the most popular QMC methods is the lattice rule, whose integration points are chosen to have a lattice structure,  low-discrepancy, and better distribution properties than random sampling \cite{Korobov,Sloan,XW}, hence, resulting in a more accurate method with  faster convergence rate. However, traditional lattice rules still have limitations when applied to high-dimensional problems. Good lattice rules almost always involve searching (cf. \cite{SH,AC}), the cost of an  exhaustive search (for $n$ fixed) grows exponentially with the dimension $d$.  Moreover, like the MC method,  the number integration points required to achieve a reasonable accuracy also increases exponentially with the dimension $d$ (i.e., the CoD phenomenon), which makes the method computationally infeasible for very high-dimensional integration problems. 
 
 To overcome the limitations of QMC  lattice rules, we first propose an improved QMC lattice rule based on a change of variables and reformulate it as 
 a tensor product rule in the transformed coordinates.  We then develop an efficient implementation algorithm, called MDI-LR,  by adapting  the multilevel dimension iteration (MDI) idea first proposed by the authors in \cite{Feng-Zhong}, for the improved QMC lattice rule.  The proposed MDI-LR algorithm optimizes the function evaluations at integration points by clustering them and sharing computations via a symbolic-function based dimension/coordinate iteration procedure.  This MDI-LR algorithm significantly reduces the computational complexity of the QMC lattice rule from an exponential growth in 
 dimension $d$ to a polynomial order $O(N^2d^3)$, where $N$ denotes the number of integration points in each (transformed) coordinate direction.
 Thus, the MDI-LR effectively overcomes the CoD and revitalizes QMC lattice rules for  high-dimensional integration. 
 
 The remainder of this paper is organized as follows. In Section \ref{sec-2}, we first briefly review the rank-one  lattice rule and its properties. In Section \ref{sec-3},  we introduce a reformulation of this lattice rule and proposed a tensor product generalization based on an affine transformation.  
 In Section \ref{sec-3a} , we introduce our MDI-LR algorithm for efficiently implementing the proposed lattice rule based a multilevel dimension iteration idea. 
  In Section \ref{sec-4}, we present extensive numerical experiments  to test the performance of the proposed MDI-LR algorithm and compare its performance with the original lattice rule and the improved lattice rule with standard implementation. The numerical experiments show that the MDI-LR algorithm is much faster and more efficient in medium and high-dimensional cases. In Section \ref{sec-5}, we numerically examine the impact of parameters appeared in  MDI-LR algorithm, including the choice of the generating vector $\mathbf{z}$ for the lattice rule. In Section \ref{sec-6}, we present a detail numerical study of  the computational complexity for the MDI-LR algorithm. This is done by using regression techniques to discover the relationship between CPU time and dimension  $d$. Finally, the paper is concluded with a summary given in Section \ref{sec-7}.

\section{Preliminaries}\label{sec-2}
In this section, we first briefly recall some basic materials about Quasi-Monte Carlo (QMC) lattice rules for evaluating integrals \eqref{eq1.1} and their properties, they will set stage for us to introduce our fast implementation algorithm in the later sections. 
 
\subsection{Quasi-Monte Carlo lattice rules}\label{sec-2.1}
 Lattice rules are a class of Quasi Monte Carlo (QMC) methods which were first introduced by Korobov in \cite{Korobov} to approximate \eqref{eq1.1} with periodic integrand $f$. A lattice rule is an equal-weight cubature rule whose cubature  points are those points of an integration lattice that lie in the half-open unit cube $[0, 1)^d$. Every lattice point set includes the origin.  The projection of the lattice
 points onto each coordinate axis are equally spaced.  Essentially,  the integral  is approximated in each coordinate direction  by a rectangle rule (or a trapezoidal rule if the integrand is periodic).
The simplest lattice rules are called rank-one lattice rules, they use a lattice point set generated by multiples of a single generator vector, which are defined as follows.
\begin{definition}[rank-one lattice rule]
	An n-point rank-one lattice rule in $d$-dimensions, also known as the method of good lattice points, is a QMC method with cubature points
	\begin{equation}\label{eq2.1}
		\mathbf{x}_i=\Big\{\frac{i\mathbf{z}}{n}\Big\},  \qquad  i=0,1,\cdots,n-1,
	\end{equation}
	where $\mathbf{z}\in \mathbb{Z}^d$, known as the generating vector, is a $d$-dimensional integer
	vector having no factor in common with n, and the braces operator $\{\cdot\}$ takes the fractional part of the input vector.
\end{definition}

Every lattice rule can be written as a multiple
sum involving one or more generating vectors. {\em The minimal number of generating vectors required to generate a lattice rule is known as the rank of the rule.} Besides rank-one lattice rules which have only one generating vector, there are also lattice rules having rank up to $d$.

$f$ is said to have an absolutely convergent Fourier series expansion if 
\begin{equation}\label{eq2.2}
	f(\mathbf{x})=\sum_{\mathbf{h}\in Z^d}\hat{f}(\mathbf{h})e^{2\pi i\mathbf{h}\cdot \mathbf{x}},\qquad  i=\sqrt{-1}, 
\end{equation}
where the Fourier coefficient is defined as 
$$\hat{f}(\mathbf{h})=\int_\Omega f(\mathbf{x})e^{-2\pi i\mathbf{h}\cdot \mathbf{x}} d\mathbf{x}.$$

The following theorem gives two characterizations for the error of the lattice rules (cf. \cite[Theorem 1]{Sloan}
and \cite[Theorem 5.2]{DKS13}).

\begin{theorem}
	Let $Q_{n,d}$ denote a lattice rule
	(not necessarily rank-one) and let $\mathcal{L}$ denote the associated integration lattice. If $f$ has an absolutely convergent Fourier series \eqref{eq2.2}, then
	\begin{equation}
		Q_{n,d}(f)-I_d(f)=\sum_{\mathbf{h}\in \mathcal{L}^{\perp}\setminus\{\mathbf{0}\}}\hat{f}(\mathbf{h}),
	\end{equation}
	where $\mathcal{L}^{\perp}:=\{\mathbf{h}\in\mathbb{Z}^{d}:\mathbf{h}\cdot\mathbf{x}\in\mathbb{Z} \quad\forall \mathbf{x}\in\mathcal{L}\}$ is the dual lattice associated with $\mathcal{L}$.
\end{theorem}

\begin{theorem}
	Let $Q_{n,d}$ denote a rank-one lattice rule with a generating vector $\mathbf{z}$.  If $f$ has an absolutely convergent Fourier series \eqref{eq2.2}, then
	\begin{equation}\label{eq2.3}
		Q_{n,d}(f)-I_d(f)=\sum_{\begin{array}{c}
				\mathbf{h}\in \mathbb{Z}^{d}\setminus\{\mathbf{0}\}\\
				\mathbf{h}\cdot\mathbf{z}\equiv 0\,  (\mbox{\rm mod}\,  n)  
		\end{array}} \hat{f}(\mathbf{h}).
	\end{equation}
\end{theorem}

It follows from \eqref{eq2.3} that the least upper bound of the error for the class {\color{black} $E_{\alpha}(c)$ of functions whose Fourier coefficients satisfy $|\hat{f}(\mathbf{h})|\leq\frac{c}{(\bar{h}_1\cdots\bar{h}_d)^{\alpha}}$ } (where $\bar{h}:=\max(1,|h|)$ and $\alpha>1,c>0,\mathbf{h}\neq 0$)  is given by 
\begin{equation}
	|Q_{n,d}(f)-I_d(f)|\leq c\sum_{\begin{array}{c}
			\mathbf{h}\in \mathbb{Z}^{d}\setminus\{\mathbf{0}\}\\
			\mathbf{h}\cdot\mathbf{z}\equiv 0\,  (\mbox{mod}\,  n)
	\end{array}}\frac{1}{(\bar{h}_1\cdots\bar{h}_d)^{\alpha}}.
\end{equation}

Let
\begin{equation}
	{\color{black} P_{\alpha,n,d}(\mathbf{z}):}=\sum_{\begin{array}{c}
			\mathbf{h}\in \mathbb{Z}^{d}\setminus\{\mathbf{0}\}\\
			\mathbf{h}\cdot\mathbf{z}\equiv 0 \,  (\mbox{mod}\,  n) 
	\end{array}}\frac{1}{(\bar{h}_1\cdots\bar{h}_d)^{\alpha}}. 
\end{equation}
For fixed $n$ and $\alpha$,   a good lattice point $\mathbf{z}$ is so chosen to make $P_{\alpha,n,d}(\mathbf{z})$ as small as possible. 
It ws proved by Niederreiter in \cite[Theorem 2.11]{Niederreiter_H,Niederreiter_H2} that for a prime $n$ (or a prime power) there exists a lattice point $\mathbf{z}$ such that
\begin{equation}
	P_{\alpha,n,d}(\mathbf{z})=O\biggl(\frac{(\log n)^{\alpha d}}{n^{\alpha}} \biggr).
\end{equation}
This was done by proving that 	{\color{black}$P_{\alpha,n,d}(\mathbf{z})$ }has the following expansion: 
\begin{align}\label{eq2.4}
		&P_{\alpha,n,d}(\mathbf{z})=-1+\frac{1}{n}\sum\limits_{k=0}^{n-1}\prod_{j=1}^{d} \biggl(1+\sum_{h\in\backslash\{0\}}\frac{e^{2\pi i khz_j/n}}{|h|^{\alpha}} \biggr)\\
		&\quad =-1+\frac{1}{n}\prod_{j=1}^{d} \bigl(1+2\zeta(\alpha)  \bigr)+\frac{1}{n}\sum_{k=1}^{n-1}\prod_{j=1}^{d} \biggl(1+\frac{(-1)^{\frac{\alpha}{2}+1}(2\pi)^{\alpha}}{\alpha!}B_{\alpha}\Bigl(\Bigl\{\frac{kz_j}{n} \Bigr\} \Bigr) \biggr), \nonumber
\end{align}
where
\begin{align}
	\zeta(\alpha) &:=\sum_{j=1}^{\infty}\frac{1}{j^\alpha}, \qquad \alpha >1;\\
	B_{\alpha}(\lambda)&:=\frac{(-1)^{\frac{\alpha}{2}+1}\alpha!}{(2\pi)^{\alpha}}\sum_{h\in\mathbb{Z}\backslash\{0\}}\frac{e^{2\pi i h\lambda}}{|h|^{\alpha}}, \qquad \lambda\in[0,1].
\end{align}

As expected, the performance of a lattice rule depends heavily on the choice of the generating vector $\mathbf{z}$. For large $n$ and $d$, 
an exhaustive search to find such a generating vector by minimizing some desired error criterion is practically impossible. Below we list  a few common 
strategies for constructing lattice generating vectors.

{\color{black}
	  We end this subsection by stating some well-known error estimate results. To the end,  we need to introduce some notations. The worst-case error of a QMC rule 	$Q_{n,d}(f)$ using the point set $P \subset [0,1]^d$ in a normed space $H$ (with the norm $\|\cdot\|$) is  
	\[ 
	E_{n,d}(P):= \sup_{\|f\|\leq 1} |I_d(f)-Q_{n,d}(f)|.  
	\]
By linearity, for any function $f\in H$, we have
\[ 
\bigl|I_d(f)-Q_{n,d}(f) \bigr| \leq E_{n,d}(P) \|f\|.  
\]

For a given (shift) vector $\Delta\in[0, 1]^d$, we define the shifted lattice  $P + \Delta := \{\{\mathbf{t} + \Delta\} : t\in P\}$. 
For any  QMC lattice rule $Q_{n,d}(\cdot)$ with  the lattice point set $P$,   let $Q^{(sh)}_{n,d}(\cdot)$ denote the corresponding shifted QMC lattice rule over the lattice  $P + \Delta$.   Then, for any integrand $f\in H$, it follows from the definition of the worst-case error that 
	\[
	\bigl|I_d(f)-Q_{n,d}^{(sh)}(f) \bigr| \leq E_{n,d}(P+\Delta) \|f\|.  
	\]	
{\color{black} 
	Define the quantity $$E^{(sh)}_{n,d}(P):=  \biggl(\int_{[0,1]^d}E^2_{n,d}(P+\Delta)\,d\Delta\biggr)^{\frac12},$$ which denotes the shift-averaged worst-case error.  The following bound for the root-mean-square error was derived in \cite[Section 5.2]{DKS13}:
\[ 
\biggl(\mathbb{E}\bigl|I_d(f)-Q_{n,d}^{(sh)}(f) \bigr|^2 \biggr)^{\frac12} \leq E^{(sh)}_{n,d}(P)\|f\|.    
\]
where the expectation $\mathbb{E}$ is taken over the random shift $\Delta$ which is uniformly
distributed over $[0, 1]^d$.
The shift-averaged worst-case error $E^{(sh)}_{n,d}(P)$ is often used as quality measure for randomly shifted QMC rules. For any given point set $P$, the averaging argument guarantees the existence of at least one shift $\Delta$ for which 
	\begin{equation}
		 E_{n,d}(P+\Delta)\leq E_{n,d}^{(sh)}(P).
	\end{equation}
}
	
In the case of rank-one  QMC lattice rule with the generating vector $\mathbf{z}$, we use $E_{n,d}(\mathbf{z})$ and $E^{(sh)}_{n,d}(\mathbf{z})$ to denote $E_{n,d}(P)$ and $E^{(sh)}_{n,d}(P)$.  
It was proved in \cite[Lemma 5.5]{DKS13} that, for any rank-one QMC lattice rule,  $\bigl[E^{(sh)}_{n,d}(\mathbf{z}) \bigr]^2$ has an explicit formula as quoted in the following theorem. 
	
	\begin{theorem}
		The shift-averaged worst-case error for a rank-one QMC attice rule in the weighted anchored or unanchored Sobolev space (see Remark \ref{rem-1.2} below for the definitions) is given by
		\begin{align}\label{eq2.12}
			\bigl[E^{(sh)}_{n,d}(\mathbf{z}) \bigr]^2&=\sum_{\emptyset\neq\nu\subseteq\{1:d\}}\gamma_{\nu}\bigg(\frac{1}{n}\sum_{k=0}^{n-1}\prod_{j\in\nu}\bigg[B_{2}\bigg( \bigg\{\frac{kz_j}{n}\bigg\} \bigg)+\beta\bigg]-\beta^{|\nu|} \bigg)\\
			&=-\prod_{j=1}^{d}(1+\gamma_j\beta)+\frac{1}{n}\sum_{k=0}^{n-1}\prod_{j=1}^{d}\bigg(1+\gamma_j\bigg[B_2\bigg(\bigg\{
			\frac{kz_j}{n}\bigg\}\bigg)+\beta\bigg]\bigg), \nonumber
		\end{align}
where $\beta = c^2-c +\frac13$  for the anchored Sobolev space and  $\beta = 0$ for unanchored Sobolev space. $\{\gamma_{\nu}\}$ are weights.  
	\end{theorem}

\begin{remark}\label{rem-1.2}
	(1) We recall that for general given weights $\{\gamma_{\nu}\}$,  the inner product  of the weighted anchored Sobolev space is defined by 
	\begin{equation}
		<f,g>_{d,\gamma}:=\sum_{\nu\subseteq\{1, 2,\cdots, d\}}\gamma_{\nu}^{-1}\int_{[0,1]^{|\nu|}}\frac{\partial^{|\nu|}}{\partial \mathbf{x}_{\nu}}f(\mathbf{x}_{\nu};  c )\frac{\partial^{|\nu|}}{\partial \mathbf{x}_{\nu}}g(\mathbf{x}_{\nu};  c) d\mathbf{x}_{\nu}.
	\end{equation}
{\color{black} where the sum is over
	all subsets $\nu\subseteq\{1,2,\cdots,d\}$, including the empty set, while for $\mathbf{x} \in[0, 1]^d$ the
	symbol $\mathbf{x}_{\nu}$ denotes the set of components $x_j$ of $\mathbf{x}$ with $j\in \nu$, and $(\mathbf{x}_{\nu}; c)$
	denotes the vector obtained by replacing the components of $\mathbf{x}$ for $j \notin \nu$ by $c\in [0,1]$ which is called the `anchor' value. The partial derivative $\frac{\partial^{|\nu|}}{\partial \mathbf{x}_{\nu}}$ denotes the mixed first partial derivative with respect
	to the components of $\mathbf{x}_{\nu}$.} 

 (2) The inner product for the weighted unanchored  Sobolev space is defined by 
 \begin{align}
 	<f,g>_{d,\gamma} &:=\sum_{\nu\subseteq\{1:d\}}\gamma_{\nu}^{-1}\int_{[0,1]^{|\nu|}}\bigg(\int_{[0,1]^{d-|\nu|}}\frac{\partial^{|\nu|}}{\partial \mathbf{x}_{\nu}}f(\mathbf{x})d\mathbf{x}_{-\nu}\bigg) \\
  &	\hskip 1.5in 
 	\times\bigg(\int_{[0,1]^{d-|\nu|}}\frac{\partial^{|\nu|}}{\partial \mathbf{x}_{\nu}}g(\mathbf{x})d\mathbf{x}_{-\nu}\bigg)d\mathbf{x}_{\nu}, \nonumber
 \end{align}
where $\mathbf{x}_{-\nu}$ stands for the vector consisting of the remaining components of the $d$-dimensional vector $\mathbf{x}$ that
are not in $\mathbf{x}_{\nu}$.
\end{remark}

\subsection{Examples of good rank-one lattice rules}

The first example is the Fibonacci lattice, we refer the reader to \cite{DKS13} for the details.

\begin{example}[Fibonacci lattice]\label{ex1}
	Let $\mathbf{z} = (1, F_k)$ and $n = F_{k+1}$, where
	$F_k$ and $F_{k+1}$ are consecutive Fibonacci numbers. Then the resulting two-dimensional lattice set generated by  $\mathbf{z}$ is called a Fibonacci lattice.  
\end{example}

Fibonacci lattices in 2-d  have a certain optimality property, but there is no obvious generalization to higher dimensions that
retains the optimality property (cf. \cite{DKS13} ).

The second example is so-called Korobov lattices, we refer the reader to \cite{NM,Niederreiter_1} for the details.

\begin{example}[Korobov lattice]\label{ex2}
	Let $a$ be an integer satisfying $1 \leq a \leq n-1$ and $\gcd(a, n) = 1$ and 
	\begin{equation*}
		\mathbf{z}=\mathbf{z}(a):=(1,a,a^2,\cdots,a^{d-1})\mod n.
	\end{equation*}
	Then the resulting $d$-dimensional lattice set generated by  $\mathbf{z}$ is called a  Korobov  lattice.  
\end{example}

It is easy to see that  there are (at most) $n-1$ choices for the Korobov parameter $a$, which leads to (at most) $n-1$ choices for the generating vector $\mathbf{z}$. Thus it is feasible in practice to search through the (at most) $n-1$ choices and take the one that fulfills the desired error criterion 
such as the one that minimizes $P_{\alpha,n,d}(\mathbf{z})$, and  \eqref{eq2.4} allows $P_{\alpha,n,d}(\mathbf{z})$ to be computed in $O(dn^2)$ operations (cf.  \cite{XW}). 
 
 	 The last example is called the CBC lattice which is based on the component-by-component construction (cf. \cite{Sloan_A}).   
 
\begin{example}[CBC lattice]\label{ex3b}
	Let $\mathbb{N}_n:=\{z\in \mathbb{Z}:1\leq z\leq n-1 \text{ and } \gcd(z,n)=1 \}$. Given $n, d$ and weights as in $\gamma_{\nu}$ in \eqref{eq2.12},  define generating vector $\mathbf{z} = (z_1, z_2,\cdots, z_d)$ component-wise  as follows.
\begin{itemize}
	\item[\rm (i)]  Set $z_1 = 1$.
	
	\item[\rm (ii)] With $z_1$ held fixed, choose $z_2$ from $\mathbb{N}_n$ to minimize $[E^{(sh)}_{n,d}((z_1,z_2))]^2$  in $2$-d.
	
	\item[\rm (iii)] With $z_1, z_2$ held fixed, choose $z_3$ from $\mathbb{N}_n$ to minimize $[E^{(sh)}_{n,d}((z_1,z_2,z_3))]^2$  in $3$-d.
	
	
\item[\rm (iv)] repeat the above process until all $\{z_j\}_{j=1}^d$ are determined. 
\end{itemize}
\end{example}

With general weights $\{\gamma_{\nu}\}$, the cost of the CBC algorithm is prohibitively expensive, thus in practice some special structure is always adopted, among them, product weights, order-dependent weights, finite-order weights, and POD (product and order-dependent) weights are commonly used.  In each of the $d$ steps of the CBC algorithm,  the search space $\mathbb{N}_n$ has cardinality $n-1$. Then the overall search space for the CBC algorithm is reduced to a size of order $O(dn)$ (cf. \cite[page 11]{PHF}).  Hence, this provides a feasible way of constructing a generating vector $\mathbf{z}$.

Figure \ref{fig.1} shows a two-dimensional lattice with 81 points, the corresponding generating vectors are (1, 2), (1, 4) and (1, 7) respectively. Figure \ref{fig.2} shows a three-dimensional lattice with 81 points, and the corresponding generating vectors are respectively (1, 2, 4), (1, 4, 16) and (1, 7, 49).

\begin{figure}[H]
	\centerline{
		\includegraphics[width=1.6in,height=1.5in]{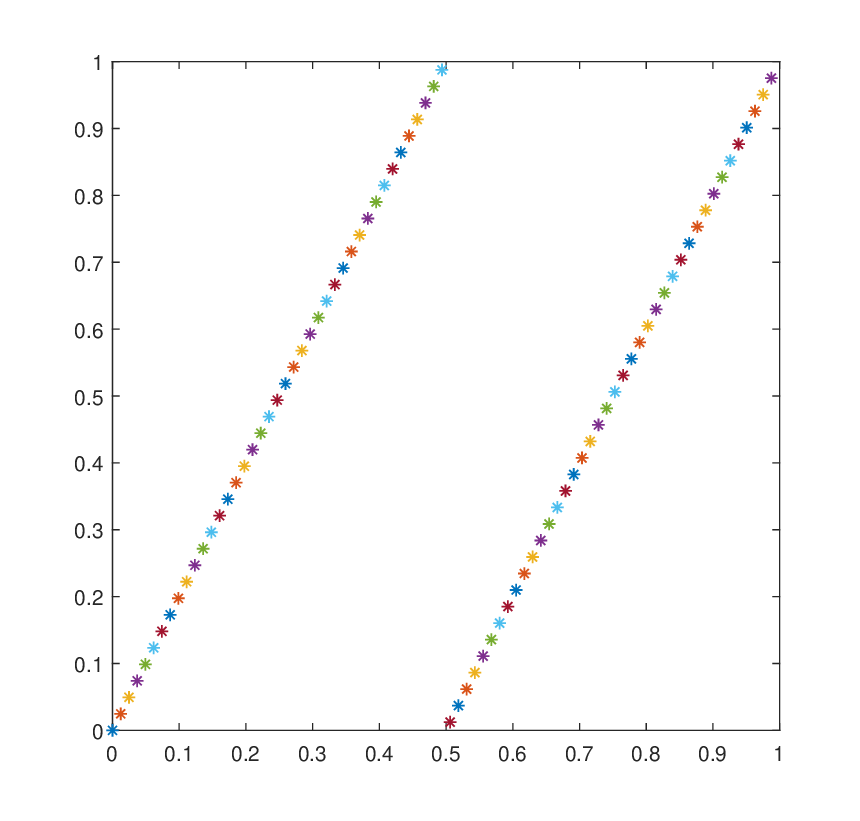}
		\includegraphics[width=1.6in,height=1.5in]{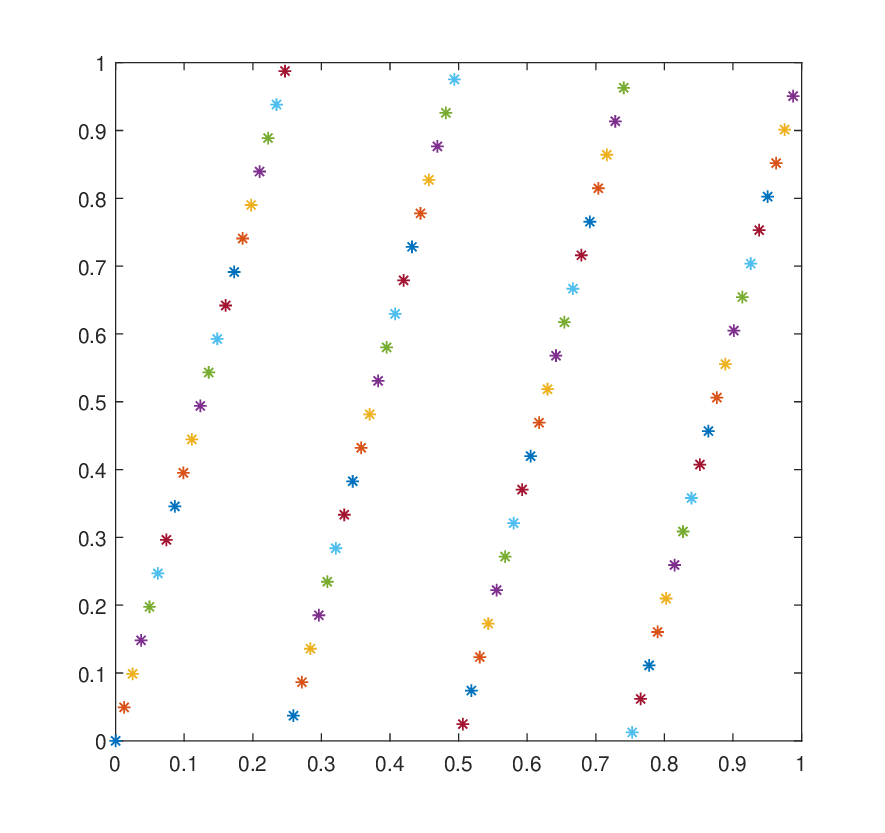}
		\includegraphics[width=1.6in,height=1.5in]{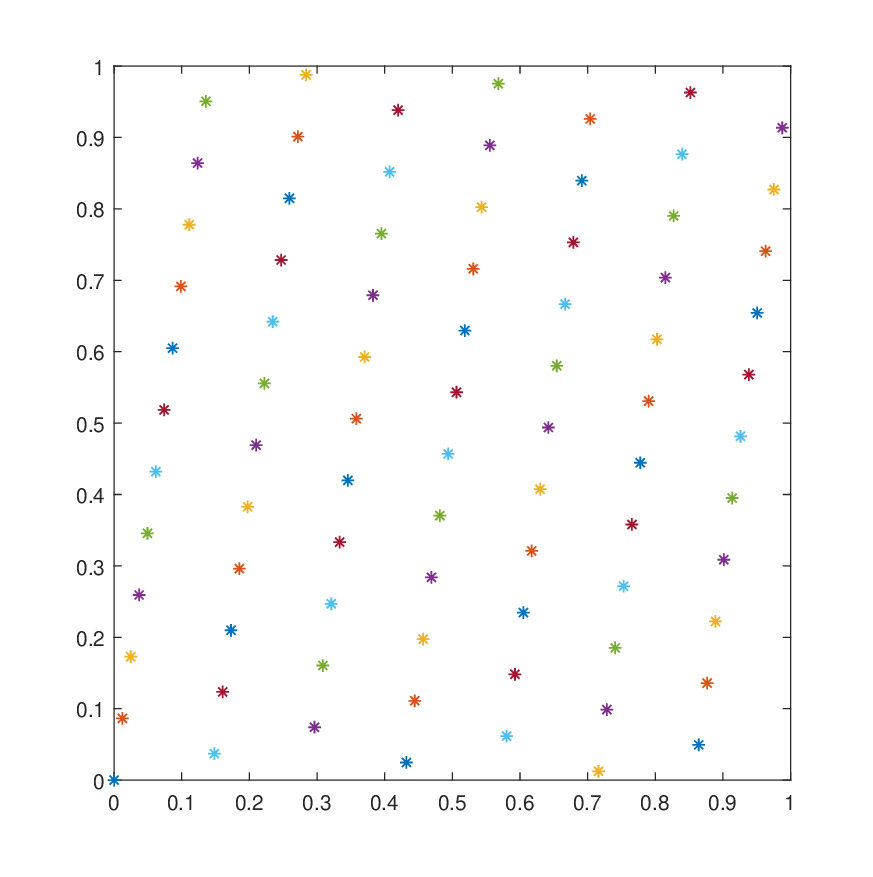}
	}
	\caption{$81$-point lattice with generating vectors $(1, 2), (1, 4)$, and $(1, 7)$.}	\label{fig.1}   
\end{figure}

\begin{figure}[H]
	\centerline{
		\includegraphics[width=1.7in,height=1.6in]{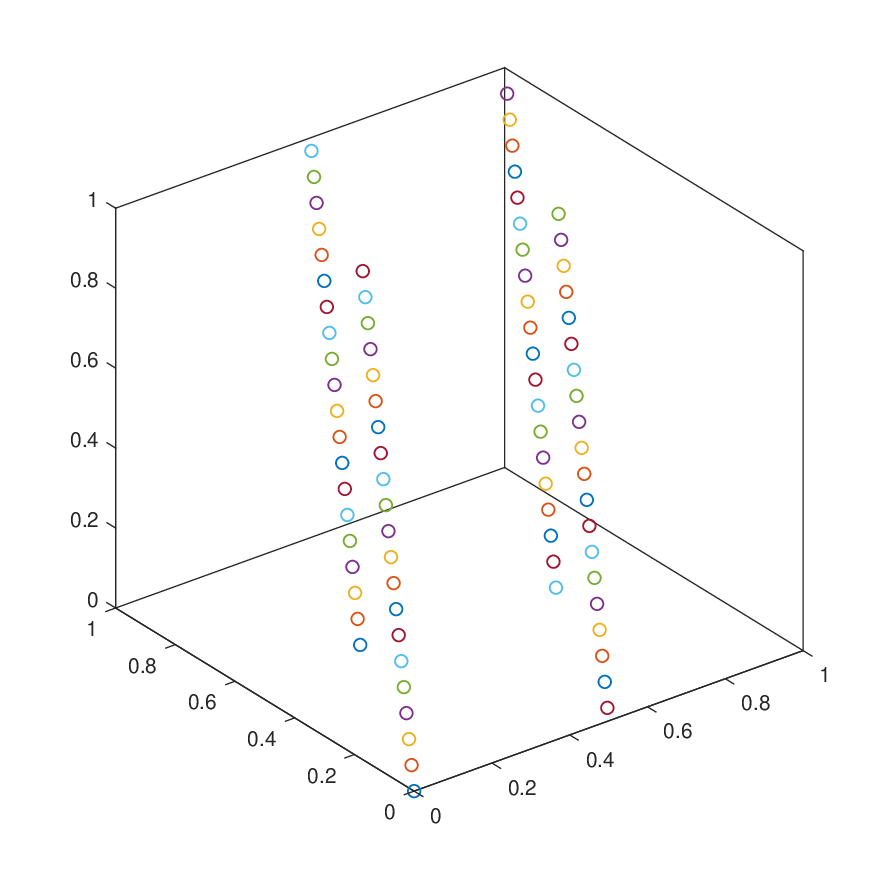}
		\includegraphics[width=1.7in,height=1.6in]{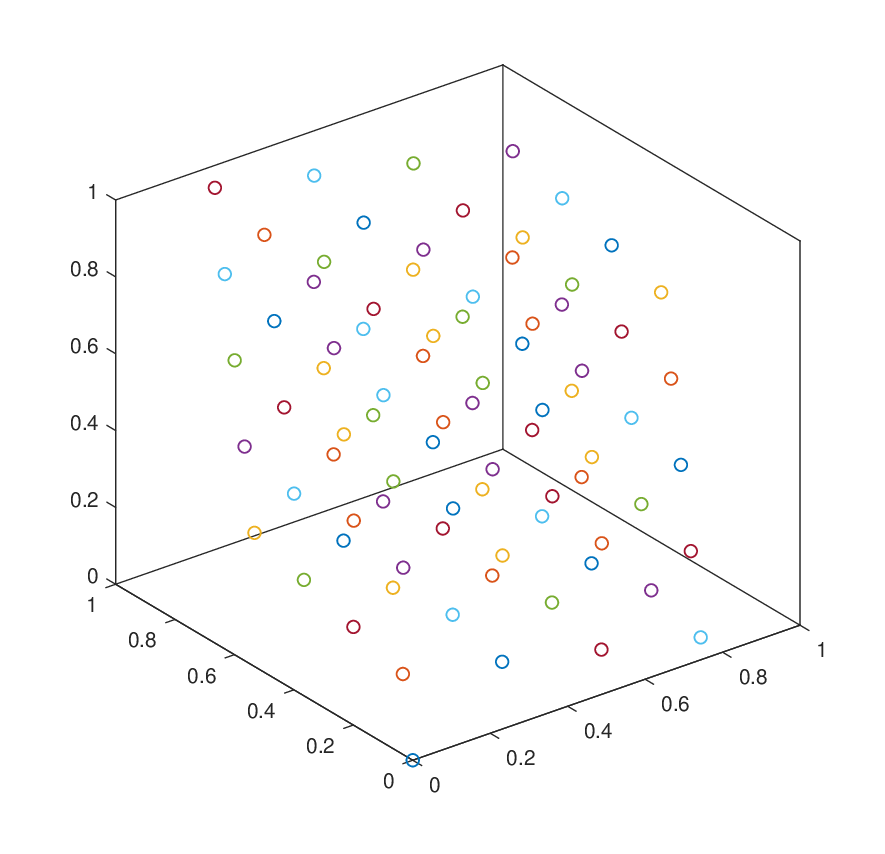}
		\includegraphics[width=1.7in,height=1.6in]{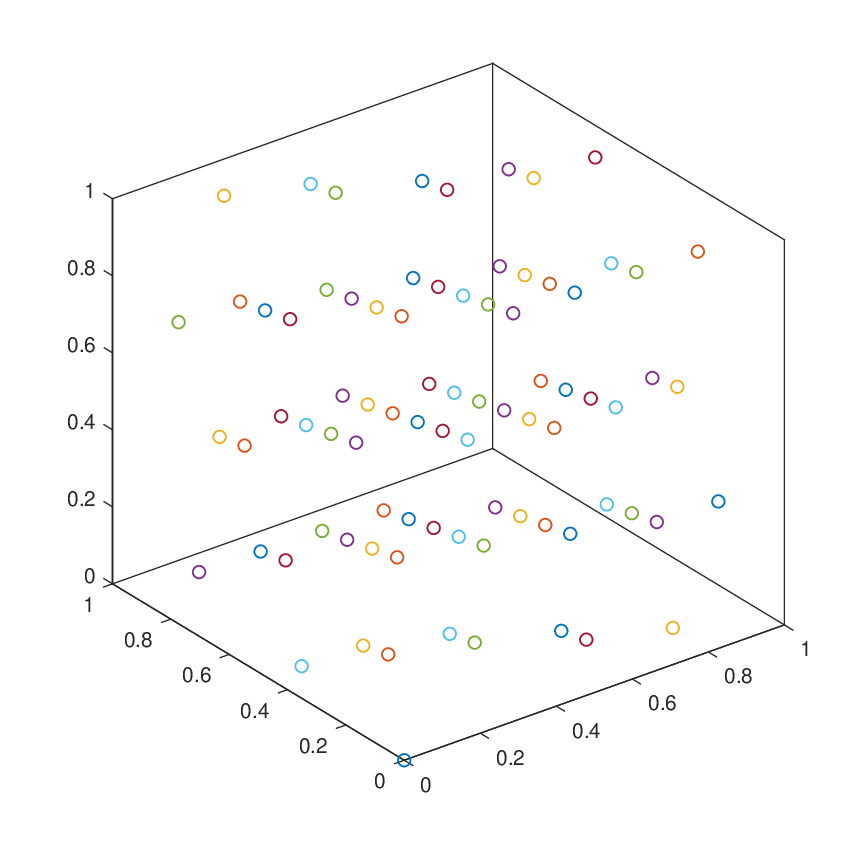}
	}
	\caption{$81$-point lattice with generating vectors $(1, 2, 4), (1, 4, 16)$, and $(1, 7, 49)$.}	\label{fig.2}   
\end{figure}

\section{Reformulation of lattice rules}\label{sec-3a}
Clearly, the lattice point set of each  QMC lattice rule has some pattern or structure.  Indeed,  one main goal of this section is precisely to describe the pattern. We show that a lattice rule almost has a tensor product reformulation viewed in an appropriately  transformed coordinate space via an affine transformation.  This discovery allows us to introduce a tensor product rule as an improvement to the original QMC lattice rule. More importantly,  the reformulation lays an important jump pad for us to develop 
an efficient and fast implementation algorithm (or solver), called the MDI-LR algorithm, based on the idea of multilevel dimension iteration  \cite{Feng-Zhong}, for evaluating the QMC lattice rule \eqref{eq1.2}. 

\subsection{Construction of affine coordinate transformations}
From Figure \ref{fig.1}-\ref{fig.2}, we see that the distribution of lattice points are on lines/planes which are not parallel to the coordinate axes/planes,
however, those lines/planes are parallel to each other, this observation suggests that they can be made to parallel to the coordinate axes/planes via affine 
transformations.  Below we prove that is indeed the case and explicitly construct such an affine transformation for a given QMC lattice rule.

\begin{theorem}\label{the1}
	Let $\mathbf{z}\in \mathbb{Z}^d$ and $\mathbf{x}_j=\bigl\{\frac{j\mathbf{z}}{n}\big\}$ for $j=0,1,2,\cdots, n-1$ denote the rank-one QMC lattice rule  point set.
	Define
	\begin{equation}\label{eq3.1}
		A=\left(\begin{array}{ccccccc}
			\frac{1}{z_1}&-\frac{1}{z_2}&0&\cdots&0&0\\
			0&\frac{1}{z_2}&-\frac{1}{z_3}&\cdots&0&0\\
			\vdots&\vdots&\vdots&\vdots&\vdots&\vdots\\
			0&0&0&\cdots&\frac{1}{z_{d-1}}&-\frac{1}{z_d}\\
			0&0&0&\cdots&0&1\end{array}\right)  \quad \mbox{and } \quad 
		\mathbf{b}=\left(\begin{array}{c}
			0\\0\\\vdots\\-\bigl\{\frac{nx_d}{z_d} \bigr\}\cdot\frac{z_d}{n}	
		\end{array}\right).
	\end{equation} 
	Notice that $A\in \mathbb{R}^{d\times d}$ and $\mathbf{b}\in \mathbb{R}^d$.  
	Then $\mathbf{y}_j:=abs(A\mathbf{x}_j+\mathbf {b}), j=0,1,\cdots,n-1$ form a Cartesian grid in the new coordinate system, where $abs(\mathbf{y})$ defines as taking  the absolute value of each component in the vector $\mathbf{y}$. 
\end{theorem}

\begin{proof}
	By the definition of $\mathbf{x}_j$, we have  $\mathbf{x}_j=\big(\{\frac{jz_1}{n}\},\{\frac{jz_2}{n}\},\cdots,\{\frac{jz_d}{n}\}\big)^{\intercal}$. A direct computation yields
	\begin{align}
		\mathbf{y}_j &= abs(A\mathbf{x}_j+\mathbf {b})
			=abs\left(\begin{array}{cccc} 
			\frac{1}{z_1}\{\frac{jz_1}{n}\}-\frac{1}{z_2}\{\frac{jz_2}{n}\}\\\frac{1}{z_2}\{\frac{jz_2}{n}\}-\frac{1}{z_3}\{\frac{jz_3}{n}\}\\\vdots\\\frac{1}{z_{d-1}}\{\frac{jz_{d-1}}{n}\}-\frac{1}{z_d}\{\frac{jz_d}{n}\}\\\{\frac{jz_d}{n}\}-\{ \frac{n}{z_d} \{\frac{jz_d}{n} \} \}\cdot\frac{z_d}{n}	
		\end{array}\right)
	\end{align}
{\color{black} Recall that $\{x\}$  and  $\lfloor x\rfloor$ denote respectively the fractional and integer parts of the number $x$.}  Because $$\frac{1}{z_i}\Big\{\frac{jz_i}{n}\Big\}=\frac{1}{z_i}\Big(\frac{jz_i}{n}-\Big\lfloor\frac{jz_i}{n}\Big\rfloor\Big)=\frac{j}{n}-\frac{1}{z_i} \Big\lfloor\frac{jz_i}{n} \Big\rfloor,$$  then $$\frac{1}{z_{i-1}} \Big\{\frac{jz_{i-1}}{n} \Big\}-\frac{1}{z_i} \Big\{\frac{jz_i}{n} \Big\}=\frac{1}{z_i} \Big\lfloor\frac{jz_i}{n} \Big\rfloor-\frac{1}{z_{i-1}} \Big\lfloor\frac{jz_{i-1}}{n} \Big\rfloor,$$ and
	\begin{equation}\label{eq13}
		\mathbf{y}_{j}={abs}\left(\begin{array}{cccc}
			\frac{1}{z_1}\{\frac{jz_1}{n}\}-\frac{1}{z_2}\{\frac{jz_2}{n}\}\\\frac{1}{z_2}\{\frac{jz_2}{n}\}-\frac{1}{z_3}\{\frac{jz_3}{n}\}\\\vdots\\\frac{1}{z_{d-1}}\{\frac{jz_{d-1}}{n}\}-\frac{1}{z_d}\{\frac{jz_d}{n}\}\\\{\frac{jz_d}{n}\}-\{\frac{\{\frac{jz_d}{n}\}}{\frac{z_d}{n}}\}\cdot\frac{z_d}{n}	
		\end{array}\right)={abs}\left(\begin{array}{cccc}
			\frac{1}{z_2}\lfloor\frac{jz_2}{n}\rfloor-\frac{1}{z_{1}}\lfloor\frac{jz_{1}}{n}\rfloor\\\frac{1}{z_3}\lfloor\frac{jz_3}{n}\rfloor-\frac{1}{z_{2}}\lfloor\frac{jz_{2}}{n}\rfloor\\\vdots\\\frac{1}{z_d}\lfloor\frac{jz_d}{n}\rfloor-\frac{1}{z_{d-1}}\lfloor\frac{jz_{d-1}}{n}\rfloor\\\frac{z_d}{n}\lfloor j-\frac{n}{z_d}\lfloor\frac{jz_d}{n}\rfloor\rfloor
		\end{array}\right).
	\end{equation}
	It is easy to check that 
	\begin{equation}\label{eq14}
		\frac{1}{z_i} \Bigl\lfloor \frac{jz_i}{n} \Bigr\rfloor-\frac{1}{z_{i-1}} \Bigl\lfloor \frac{jz_{i-1}}{n} \Bigr\rfloor=\left\{ \begin{array}{lcl}
			0,&0\leq j <\frac{n}{z_i}\\
			\frac{1}{z_i},&\frac{n}{z_i}\leq j <\frac{n}{z_{i-1}}\\
			\frac{1}{z_{i-1}}-\frac{1}{z_i},& \frac{n}{z_{i-1}} \leq j <\frac{2n}{z_i}\\
			\frac{2}{z_{i}}-\frac{1}{z_{i-1}},& \frac{2n}{z_{i}} \leq j <\frac{2n}{z_{i-1}}\\
			\frac{2}{z_{i-1}}-\frac{2}{z_{i}},& \frac{2n}{z_{i-1}} \leq j <\frac{3n}{z_{i}}\\
			\vdots&\vdots\\		
		\end{array} \right..
	\end{equation}

On the other hand, 	let
	\begin{align*}
		&	\Gamma_{n_1}^{1}:=\Big\{y_{s_1}\, | y_{s_1}={abs}\bigg(\frac{1}{z_2}\Big\lfloor\frac{iz_2}{n} \Big\rfloor-\frac{1}{z_1} \Big\lfloor\frac{iz_1}{n} \Big\rfloor \bigg),i=0,1,\cdots,n-1,s_1=0,1,\cdots,n_1-1\Big\},\\
		&	\Gamma_{n_2}^{1}:=\Big\{y_{s_2}\, | y_{s_2}={abs}\bigg(\frac{1}{z_3}\Big\lfloor\frac{iz_3}{n}\Big\rfloor-\frac{1}{z_2}\Big\lfloor\frac{iz_2}{n}\Big\rfloor\bigg),i=0,1,\cdots,n-1,s_2= 0,1,\cdots,n_2-1\Big\},\\
		&\quad 	\vdots \nonumber \\
		& \Gamma_{n_{d-1}}^{1}:=\Big\{y_{s_{d-1}}\, | y_{s_{d-1}}={abs} \bigg(\frac{1}{z_d} \Big\lfloor\frac{iz_d}{n} \Big\rfloor-\frac{1}{z_{d-1}} \Big\lfloor\frac{iz_{d-1}}{n} \Big\rfloor\bigg), i=0,1,\cdots,n-1, \atop \hskip 2.8 in s_{d-1}=0,1,\cdots,n_{d-1}-1\Big\},\\
		&	\Gamma_{n_d}^{1}:=\Big\{y_{s_d}\, | y_{s_d}={abs}\bigg(\frac{z_d}{n} \Big\lfloor i-\frac{n}{z_d}\Big\lfloor\frac{iz_d}{n} \Big\rfloor \Big\rfloor \bigg),i=0,1,\cdots,n-1,  s_d=0,1,\cdots,n_d-1\Big\},\\
		&
\Gamma_n^d:= \Gamma_{n_1}^{1}\otimes\Gamma_{n_2}^{1}\otimes\cdots\otimes\Gamma_{n_d}^{1}, 
	\end{align*}
where
	\begin{equation}\label{eq3.5}
			n_i=\frac{{\rm lcm}(z_i,z_{i+1})}{\min(z_i,z_{i+1})}, \quad i=1,2,\cdots,d-1; 
			\qquad   n_d=\Bigl\lceil\frac{n}{ z_d}\Bigr\rceil,  
	\end{equation} 
	and ${\rm lcm}$ represents the least common multiple.
	
For any $\mathbf{y}_k =(y_{s_1},y_{s_2},\cdots,y_{s_d})^\intercal \in\Gamma_{n}^{d}$,  we have  $k =s_1+s_2n_1+s_3n_1n_2+\cdots+s_dn_{1}n_2\cdots n_{d-1}$. 
Since $s_1=0,1,\cdots,n_1,\cdots,s_d=0,1,\cdots,n_d$, then $k=0,1,\cdots,n_1,\cdots,$ $(n_1n_2\cdots n_d)$. For $\mathbf{y}_{j}$ in the set described by  \eqref{eq13}, $\forall j=1,2,\cdots,n$, we get 
	\begin{equation}
		\mathbf{y}_{j} ={abs}\left(\begin{array}{cccc}
			\frac{1}{z_2}\lfloor\frac{jz_2}{n}\rfloor-\frac{1}{z_{1}}\lfloor\frac{jz_{1}}{n}\rfloor\\\frac{1}{z_3}\lfloor\frac{jz_3}{n}\rfloor-\frac{1}{z_{2}}\lfloor\frac{jz_{2}}{n}\rfloor\\\vdots\\\frac{1}{z_d}\lfloor\frac{jz_d}{n}\rfloor-\frac{1}{z_{d-1}}\lfloor\frac{jz_{d-1}}{n}\rfloor\\\frac{z_d}{n}\lfloor k-\frac{n}{z_d}\lfloor\frac{jz_d}{n}\rfloor\rfloor
		\end{array}\right).  
\end{equation}
Let $y_{i_1}:=abs(\frac{1}{z_2}\lfloor\frac{jz_2}{n}\rfloor-\frac{1}{z_{1}}\lfloor\frac{jz_{1}}{n}\rfloor)$, it follows that there exists an $s_1$ such that $s_1=n_1-\lfloor \frac{j}{n_1}\rfloor $ , resulting in $y_{i_1}=y_{s_1}\in\Gamma_{n_1}^{1}$. In the same way, $y_{i_2}\in\Gamma_{n_2}^{1},\cdots,y_{i_d}\in\Gamma_{n_d}^{1}$.  Therefore, we conclude that $\mathbf{y}_{j}=(y_{i_1},y_{i_2},\cdots,y_{i_d})^{\intercal}\in \Gamma_{n}^{d}$,
that is, the transformed lattice points have the Cartesian product structure. 
%
	\end{proof}

\begin{lemma}\label{lem-3.2}
	Let $\mathbf{x}_j=\bigl\{\frac{j\mathbf{z}}{n}\big\}$ for $j=1,2,\cdots, n-1$ denote the Korobov lattice  point set, that is $\mathbf{z}=(1,a,a^2,\cdots,a^{d-1})$, $1\leq a\leq n-1$ and $\gcd(a,n)=1$.  Then $\mathbf{y}_j:=abs(A\mathbf{x}_j+\mathbf {b}), j=0,1,\cdots,n-1$ satisfies the conclusion of Theorem \ref{the1}. Moreover,  if $a=[n]^{\frac{1}{d}}$, then the number of points in each direction of the lattice set $\Gamma_{n}^{d}$ is the same, that is, $n_1=n_2\cdots=n_d=a$. 
\end{lemma}

\begin{proof}
From theorem \ref{the1} we have 
\begin{equation}\label{eq3.7}
	\mathbf{y}_j=abs\left(\begin{array}{cccc}
		\frac{1}{a}\lfloor\frac{ja}{n}\rfloor\\\frac{1}{a^2}\lfloor\frac{ja^2}{n}\rfloor-\frac{1}{a}\lfloor\frac{ja}{n}\rfloor\\\vdots\\\frac{1}{a^{d-1}}\lfloor\frac{ja^{d-1}}{n}\rfloor-\frac{1}{a^{d-2}}\lfloor\frac{ja^{d-2}}{n}\rfloor\\\frac{a^{d-1}}{n}\lfloor j-\frac{n}{a^{d-1}}\lfloor\frac{ja^{d-1}}{n}\rfloor\rfloor
	\end{array}\right),
\end{equation}
and 
	\begin{align*}
		\Gamma_{n_1}^{1}&=\Big\{y_{s_1}\, | y_{s_1}={abs}\bigg(\frac{1}{a}\Big\lfloor\frac{ja}{n}\Big\rfloor\bigg),j=0,1,\cdots,n-1,s_1={0,1},\cdots,n_1-1\Big\}\\
	&=\Big\{0,\frac{1}{a},\frac{2}{a},\cdots,\frac{a-1}{a}\Big\},\\
		\Gamma_{n_2}^{1}&=\Big\{y_{s_2}\, | y_{s_2}={abs}\bigg(\frac{1}{a^2}\Big\lfloor\frac{ja^2}{n}\Big\rfloor-\frac{1}{a}\Big\lfloor\frac{ja}{n}\Big\rfloor \bigg),j=0,1,\cdots,n-1,s_2={0,1},\cdots,n_2-1\Big\}\\&=\Big\{0,\frac{1}{a^2},\frac{2}{a^2},\cdots,\frac{a-1}{a^2}\Big\},\\
	&\quad 	\vdots \nonumber \\
	\Gamma_{n_{d-1}}^{1}&=\Big\{y_{s_{d-1}}\, | y_{s_{d-1}}={abs}\bigg(\frac{1}{a^{d-1}}\Big\lfloor\frac{ja^{d-1}}{n}\Big\rfloor-\frac{1}{a^{d-2}}\Big\lfloor\frac{ja^{d-2}}{n}\Big\rfloor\bigg),j=0,1,\cdots,n-1, \\ &\hskip 3.1in s_{d-1}={0,1},\cdots,n_{d-1}-1\Big\}\\
	&=\Big\{0,\frac{1}{a^{d-1}},\frac{2}{a^{d-1}},\cdots,\frac{a-1}{a^{d-1}}\Big\},\\
	\Gamma_{n_d}^{1}&=\Big\{y_{s_d}\, | y_{s_d}={abs}\bigg(\frac{a^{d-1}}{n} \Big\lfloor j-\frac{n}{a^{d-1}}\Big\lfloor\frac{ja^{d-1}}{n}\Big\rfloor\Big\rfloor\bigg),j=0,1,\cdots,n-1,  \\
	&\hskip 2.85in s_d={0,1},\cdots,n_d-1\Big\}\\
	&=\Big\{0,\frac{a^{d-1}}{ n},\frac{2a^{d-1}}{ n},\cdots, \frac{n-a^{d-1}}{n}\Big\}.
\end{align*}

For $\mathbf{y}_{j}$ in the set described by \eqref{eq3.7}, 
	let $y_{i_1}:=abs(\frac{1}{a}\lfloor\frac{ja}{n}\rfloor)$, it follows that there exists an $s_1$ such that $s_1=n_1-\lfloor \frac{j}{n_1}\rfloor $, resulting in $y_{i_1}=\frac{s_1}{a}=y_{s_1}\in\Gamma_{n_1}^{1}$. Similarly, $y_{i_2}=\frac{s_2}{a^2}=y_{s_2}\in\Gamma_{n_2}^{1},\cdots,y_{i_d}\in\Gamma_{n_d}^{1}$.  Therefore,  we conclude that $\mathbf{y}_{j}=(y_{i_1},y_{i_2},\cdots,y_{i_d})^{\intercal}\in \Gamma_{n}^{d}$,
Obviously, the transformed lattice points have the Cartesian product structure.  Moreover, if $a=[n]^{\frac{1}{d}}$, then
\begin{align*}
 &n_i=\frac{{\rm lcm}(z_i,z_{i+1})}{\min(z_i,z_{i+1})}=\frac{z_{i+1}}{z_{i}}=a, \quad i=1,2,\cdots,d-1,\\
 &n_d=\lceil\frac{n}{z_d}\rceil=\frac{a^d}{a^{d-1}}=a.
\end{align*}
Hence,  $n_1=n_2=\cdots=n_{d-1}=n_d=a$. The proof is complete. 
\end{proof}

\begin{figure}[H]
	\centerline{
		\includegraphics[width=5.2in,height=2in]{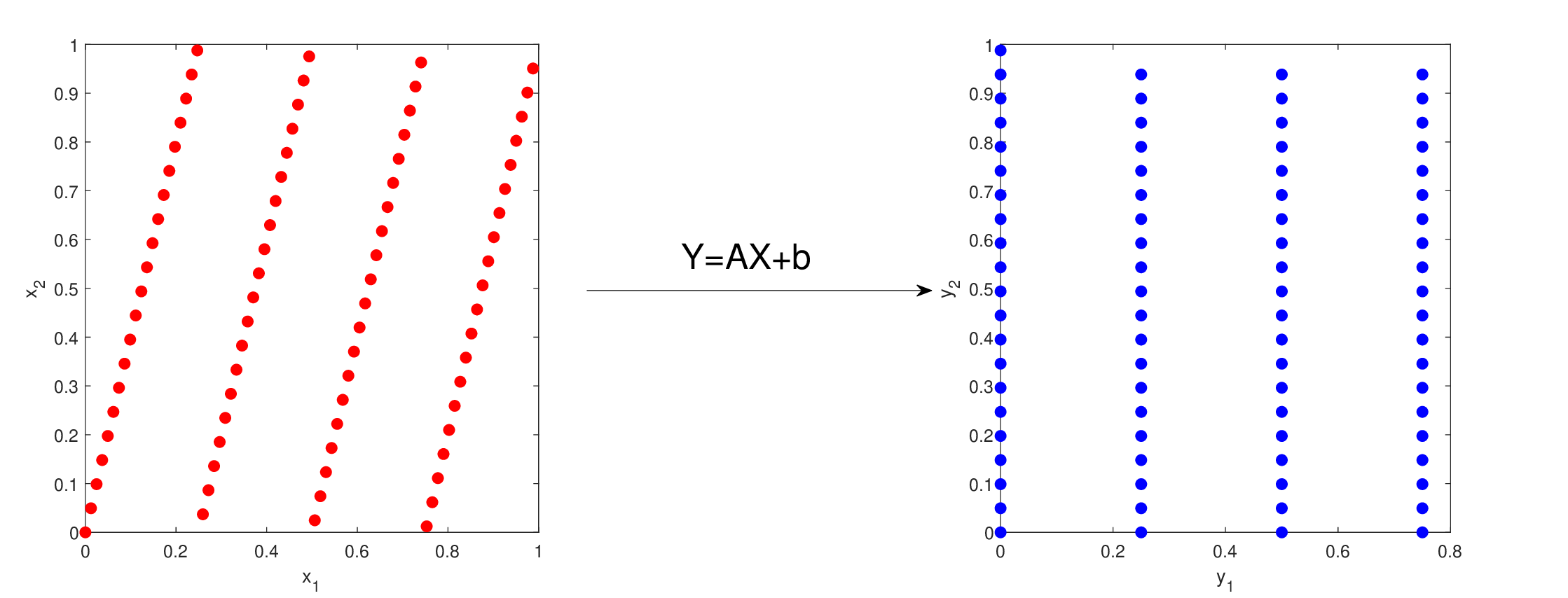}
	}
	\caption{Left: 81-point lattice with the generating vector $(1, 4)$. Right:  transformed lattice  after affine coordinate transformation.}	\label{fig.3} 
\end{figure}

Left graph of Figure \ref{fig.3} shows a 2-d example of 81-point rank-one lattice with the generating vector $(1, 4)$.  Right graph displays transformed lattice under  the affine coordinate transformation $\mathbf{y}=A\mathbf{x}+\mathbf{b}$ from $\mathbb{R}^2$ to itself, where
 \begin{equation}\label{eq2}
	A=\left(\begin{array}{cc}
		1&-\frac{1}{4}\\
		0&1\end{array}\right), \qquad
	\mathbf{b}=\left(\begin{array}{cc}
		0\\-\{\frac{81x_2}{4}\}\cdot\frac{4}{81}	
	\end{array}\right).
\end{equation}

\begin{figure}[H]
	\centerline{
		\includegraphics[width=5.5in,height=2.15in]{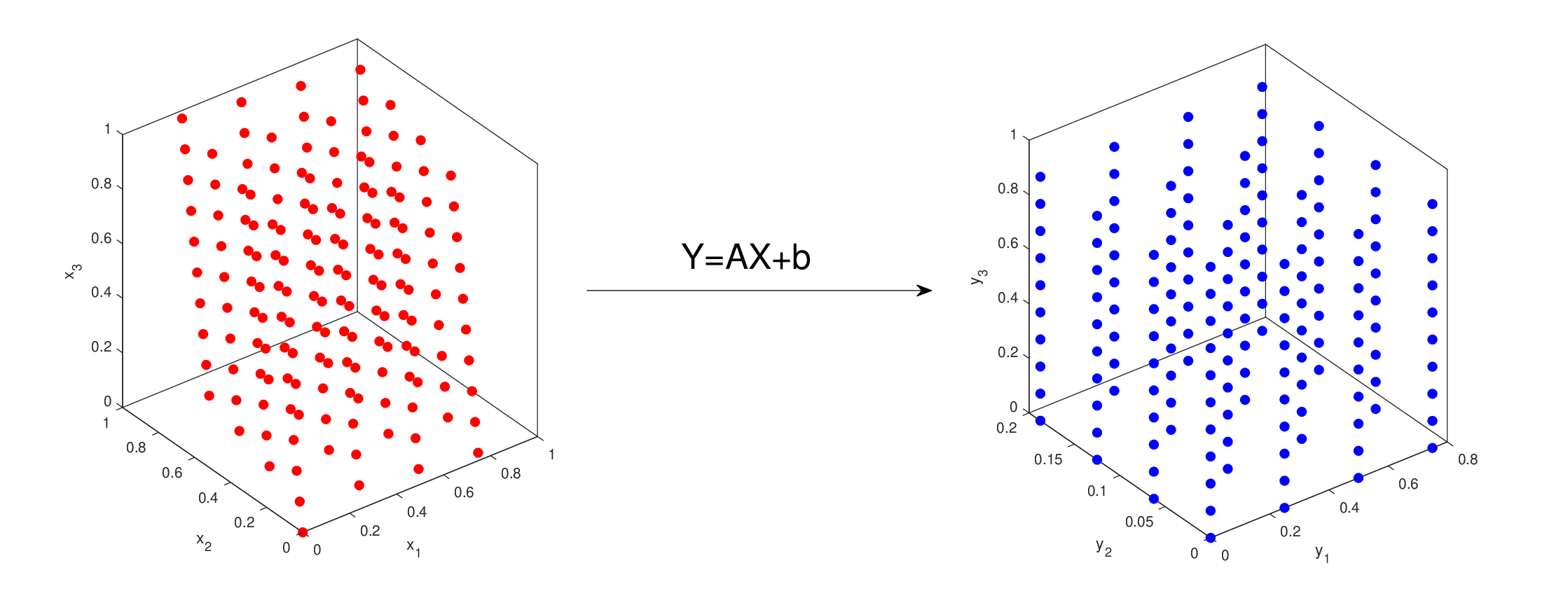}
	}
	\caption{Left: 161-point rank-one lattice with generating  vector is $(1, 4,16)$. Right: transformed lattice  after coordinate transformation.}	\label{fig.4}
\end{figure}

Figure \ref{fig.4} demonstrates a specific example in 3-d.  The left graph is a 161-point rank-one lattice with  generating vector  $(1, 4, 16)$. The right one shows  the transformed points under the affine coordinate transformation $\mathbf{y}=A\mathbf{x}+\mathbf{b}$ from $\mathbb{R}^3$ to itself, where
\begin{equation}\label{eq7}
	A=\left(\begin{array}{ccc}
		1&-\frac{1}{4}&0\\
		0&\frac{1}{4}&-\frac{1}{16}\\
		0&0&1\end{array}\right), \qquad 
	\mathbf{b}=\left(\begin{array}{ccc}
		0\\0\\-\{\frac{161x_3}{16}\}\cdot\frac{16}{161}	
	\end{array}\right).
\end{equation}
\subsection{Improved lattice rules}\label{sec-2.2.2}

From Figures \ref{fig.3} and \ref{fig.4} we see that the transformed lattice point sets do not exactly form tensor product grids because many lines miss one point.  By adding those ``missing" points which can be done systematically, we easily make them become tensor product grids in the transformed coordinate system. Since more integration points are added to the QMC lattice rule, the resulting quadrature rule is expected to be more accurate (which is supported by our numerical tests), hence, it is an improvement to the original QMC lattice rule. We also note that those added points would correspond to ghost points in the original coordinates. 

\begin{definition}[Improved QMC lattice rule]
	Let $\mathbf{z}\in \mathbb{R}^d$ and $\mathbf{x}_i=\big\{\frac{i\mathbf{z}}{n}\big\}, i=0,1,\cdots,n-1,$ be a rank-one lattice point set and $\mathbf{y}_i:=A\mathbf{x}_{i}+\mathbf {b}, i=0,1,\cdots,n-1$ for some $A\in \mathbb{R}^{d\times d}$ and $\mathbf{b}\in \mathbb{R}^d$ (which uniquely determine  an affine transformation). Suppose there exists $n^* (<<n)$ points so that together the $n+n^*$ points form a tensor product grid in the transformed coordinate system, then the QMC lattice rule obtained by using those $n+n^*$ sampling points is called {\color{black} an improved QMC lattice rule, and denoted by $	\widehat{Q}_{n,d}(f)$. }
\end{definition}

Figure \ref{fig.5} shows a 81-point (i.e., $n=81$) 2-d rank-one lattice with the generating vector $(1,7)$,  the transformed lattice (middle),  and the improved tensor product grid (right).  Three  points are added on the top, so $n^*=3$ for this example. 
\begin{figure}[H]
	\centerline{
		\includegraphics[width=5.5in,height=2.15in]{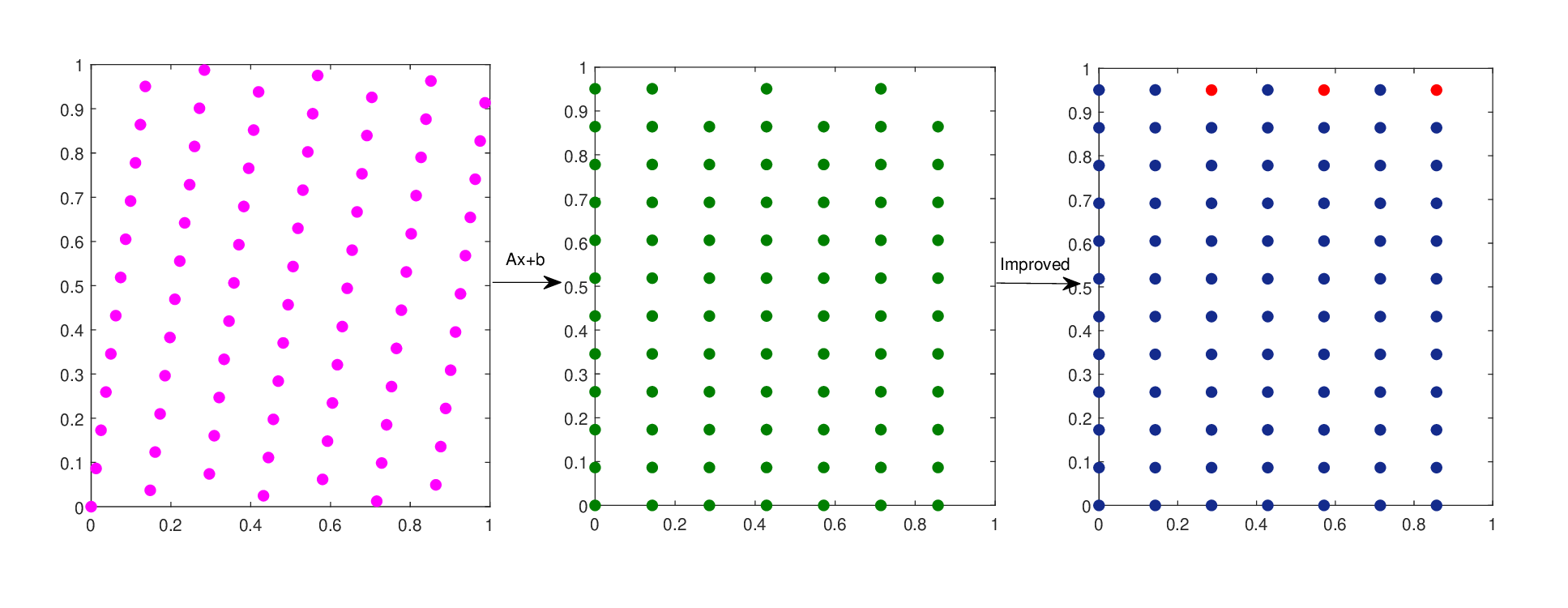}
	}
	\caption{Left: 81-point lattice with the generating vector $(1, 7)$. Middle: transformed lattice. Right:  improved tensor product grid after adding $3$ points.}	\label{fig.5}
\end{figure}

\section{The MDI-LR algorithm}\label{sec-3}
Since an improved rank-one  lattice is a tensor product grid in the transformed coordinate system and its corresponding quasi-Monte Carlo (QMC) rule is a tensor product rule with equal weight  $w=\frac{1}{n+n^*}$.  This tensor product improvement  allows us to apply the multilevel dimension iteration (MDI) approach, which was proposed by the authors in \cite{Feng-Zhong},  for a fast implementation of the original QMC lattice rule,  especially in the high dimension case. The resulting algorithm will be called the MDI-LR algorithm throughout this paper,  

\subsection{Formulation of the MDI-LR algorithm}\label{sec-2.2.3}
To formulate our MDI-LR algorithm,  we first recall the MDI idea/algorithm in simple terms (cf. \cite{Feng-Zhong}). 

For a tensor product rule, we need to compute a multi-summation with variable limits
 \[
 \sum_{i_1=1}^{n_1}\sum_{i_2=1}^{n_2}\cdots\sum_{i_d=1}^{n_d}f(\xi_{i_1},\xi_{i_2},\cdots,\xi_{i_d}), 
 \]
 which involves $n_1n_2\cdots n_d$ function evaluations for the given function $f$ if one uses the conventional approach by computing  function value at  each point independently, that inevitably leads  to the curse of dimensionality (CoD).  To make computation feasible in high dimensions,  it is imperative to save the computational cost by evaluating the summation more efficiently.  
 
 The main idea of the MDI  approach proposed in \cite{Feng-Zhong} is to compute those $n_1n_2\cdots n_d$ function values in cluster (not independently) and to compute the summation layer-by-layer based on a dimension iteration with help of symbolic computation. To the end, we write
\begin{equation}
	\sum_{i_1=1}^{n_1}\sum_{i_2=1}^{n_2}\cdots\sum_{i_d=1}^{n_d}f(\xi_{i_1},\xi_{i_2},\cdots,\xi_{i_d})=\sum_{i_{m+1}=1}^{n_{m+1}}\cdots\sum_{i_d=1}^{n_d}f_{d-m}(\xi_{i_{m+1}},\cdots,\xi_{d}),
\end{equation}
where $1\leq m<<d$ is fixed and 
\[
f_{d-m}(x_1,\cdots,x_{d-m}) :=\sum_{i_1=1}^{n_1}\cdots \sum_{i_{m}=1}^{n_m}f(\xi_{i_1},\cdots,\xi_{i_m},x_{1},\cdots,x_{d-m}).
\]
MDI approach recursively generates a sequence of symbolic functions $\{f_{d-m},f_{d-2m},$ $\cdots,f_{d-lm}\}$, each function has $m$ fewer arguments than its predecessor (because the dimension is reduced by $m$ at each iteration). As already mentioned above, the MDI approach  explores the lattice structure of the tensor product integration point set, instead of evaluating function values at all
integration points independently, it evaluates them in cluster and iteratively along $m$-coordinate directions, the function evaluation at any integration point is not completed until the last step of the algorithm is executed. In some sense,  the implementation strategy of the MDI approach is to trade large space complexity for low time complexity. That being said, however, the price to be paid by the MDI approach for the speedy evaluation of the multi-summation is that those symbolic function needs to be saved during the iteration process, which often takes up 
more computer memory.

For example, consider 2-d function $f(x_1,x_2)=x_1^2+x_1x_2+x_2^2$ and let $n_1=n_2=N$. In the standard approach, to compute the function value $f(\xi_{i_1}, \xi_{i_2})$ at an integration point  $(\xi_{i_1}, \xi_{i_2})$ , one  needs to compute  three multiplications  $\xi_{i_1}*\xi_{i_1}=\xi_{i_1}^2$, $\xi_{i_1}*\xi_{i_2}$ and $\xi_{i_2}*\xi_{i_2}=\xi_{i_2}^2$,  and  two additions. To compute $N^2$ function values, it requires a total of $3N^2$ multiplications and $3N^2-1$ additions.
On the other hand,  the first for-loop of the MDI approach generates $f_1(x)=\sum_{i_2}^{n}f(x,\xi_{i_2})$ which requires $N$ evaluations of $\xi_{i_2}*x_1$ (symbolic computations) and $N$ evaluations of $\xi_{i_2}*\xi_{i_2}$ , as well as $3(N-1)$ additions. The second for-loop generates $\sum_{i_1=1}^{N}f_1(\xi_{i_1})$ which requires $N$ evaluations of $\xi_{i_1}*\xi_{i_1}$ and $N$ evaluations of $\xi_{i_1}*\bar{\xi}_{i_2}$, as well as $3N-1$ additions.  After the second for-loop completes, we obtain the summation value. 
The computation complexity of the MDI approach consists of a total of $4N$ multiplications and $6N-4$ additions, which is much cheaper 
than the standard approach. In fact, the speedup is even more dramatic in higher dimensions. 

It is easy to see that the MDI approach can not be applied to the  QMC rule \eqref{eq1.2} because it is not in a multi-summation form. 
However,  we have showed in  Section \ref{sec-3a} that this  obstacle can be overcome by a simple affine coordinate transformation (i.e., 
change of variables) and adding a few integration points.  

Let $\mathbf{y}=A\mathbf{x}+\mathbf{b}$ denote the affine transformation, then the integral \eqref{eq1.1} is equivalent to 
\begin{equation}
	I_d(f)=\frac{1}{|A|}\int_{\widehat{\Omega}} f\bigl(A^{-1}(\mathbf{y}-\mathbf{b}) \bigr) \, d\mathbf{y} 
	=\frac{1}{|A|}\int_{\widehat{\Omega}} g(\mathbf{y})d\mathbf{y},  
\end{equation}
{\color{black} 
where $|A|$ stands for the determinant of $A\in \mathbb{R}^{d\times d}$ and 
$$
g(\mathbf{y}):=f\bigl(A^{-1}(\mathbf{y}-\mathbf{b} )\bigr), \qquad 
\widehat{\Omega}:=\bigl\{\mathbf{y}\, |  \mathbf{y}=A\mathbf{x}+\mathbf{b},\, \mathbf{x}\in \Omega \bigr\}.
$$
Then, our improved QMC rank-one lattice rule  for \eqref{eq1.2} in the $\mathbf{y}$-coordinate system takes the form
\begin{equation} \label{iQMC}
	\widehat{Q}_{n,d}(f)=\frac{1}{n+n^*}\sum_{i=0}^{n+n^*-1}f(\mathbf{x}_i)=\frac{1}{|A|}\sum_{s_1=1}^{n_1}\sum_{s_2=1}^{n_2}\cdots\sum_{s_d=1}^{n_d}g(y_{s_1},y_{s_2},\cdots,y_{s_d}).
\end{equation}

Let 
\[
J(g,\Omega):=\sum\limits_{s_1=1}^{n_1}\sum\limits_{s_2=1}^{n_2}\cdots\sum\limits_{s_d=1}^{n_d}g(y_{s_1},y_{s_2},\cdots,y_{s_d}),
\]
 Clearly, it is a multi-summation with variable limits. Thus,  we can apply the MDI approach to compute it efficiently. Before doing that, we first need to extend the MDI algorithm, Algorithm 2.3 of \cite{Feng-Zhong}, to the case of variable limits. We name  the extend algorithm as MDI$(d,g, \Omega_d, \mathbf{N}_d, m)$, which is defined as follows.
 
 \begin{algorithm}[H]
 	{\color{black}
 	\caption{MDI($d$, $g$, $\Omega, \mathbf{N}_d, m$) }
 	\label{alg:3}
 	\hspace*{0.02in} {\bf Inputs:} 
 	$d (\geq 4), g, \Omega, m (=1,2,3),\, \mathbf{N}_k=(n_1,n_2, \cdots, n_k), \, k=1,2,\cdots,d.$\\
 	\hspace*{0.02in} {\bf Output:}
 	$J=J(g,\Omega)$.
 	\begin{algorithmic}[1]
 		\State $\Omega_d=\Omega$, $g_d=g$, $\ell=[\frac{d}m]$.
 		\For $\,k=d:-m:d-\ell m$ (the index is decreased by $m$ at each iteration)
 		\State $\Omega_{d-m}=P_k^{k-m} \Omega_{k}$.
 		\State  Construct symbolic function $g_{k-m}$ by \eqref{eq2.14c} below).
 		\State  MDI$(k,g_k,\Omega_k,\mathbf{N}_k,m)$ :=MDI$(k-m,g_{k-m},\Omega_{k-m}, \mathbf{N}_{k-m},m)$.  
 		\EndFor 
 		\State $J=$ MDI$(d-\ell m, g_{d-\ell m} ,\Omega_{d-\ell m},\mathbf{N}_{d-\ell m},m)$.\\
 		\Return $J$.
 	\end{algorithmic}
 }
 \end{algorithm}
 \noindent
 Where $P^{k-m}_{k}$ denotes the natural embedding from $\mathbb{R}^k$
 to $\mathbb{R}^{k-m}$ by deleting the first $m$ components of 
 vectors in $\mathbb{R}^k$, and  
 \begin{equation}\label{eq2.14c}
 	g_{k-m}(s_1,\cdots,s_{k-m}) =  \sum_{i_1, \cdots, i_m=1}^{n_1,\cdots, n_m}  w_{i_1}w_{i_2}\cdots w_{i_m}\, g_k\bigl((\xi_1, \cdots, \xi_m, s_1, \cdots, s_{k-m}) \bigr). 
 \end{equation}
 
  \begin{remark}\
  	
 	\begin{itemize}
 		\item[(a)]  Algorithm \ref{alg:3} recursively generates a sequence of symbolic functions $\{ g_d, g_{d-m},$ $g_{d-2m}, \cdots g_{d-\ell m}  \}$, each function has $m$ fewer arguments than its predecessor. 
 		
 		\item[(b)] Since $m\leq 3$, when $d=2,3$, we simply use the underlying low dimensional QMC quadrature rules. As done in \cite{Feng-Zhong}, we name those low dimensional algorithms as 2d-MDI$(g, \Omega, \mathbf{N}_2)$ and 3d-MDI$(g, \Omega, \mathbf{N}_3)$, and  introduce the following conventions. 
 		
 		\begin{itemize}
 			\item If $k = 1$, set MDI$(k,g_k,\Omega_k,n_1,m) :=J(g_k,\Omega_k)$, which is computed by using the underlying 1-d QMC quadrature rule. 
 			\item If $k = 2$, set MDI$(k,g_k,\Omega_k,\mathbf{N}_k,m):=$ 2d-MDI$(g_{k}, \Omega_{k}, \mathbf{N}_k)$.
 			\item If $k = 3$, set MDI$(k,g_k,\Omega_k,\mathbf{N}_k,m):=$ 3d-MDI$(g_{k}, \Omega_{k}, \mathbf{N}_k)$.
 		\end{itemize}
 		We note that when $k=1,2,3$, the parameter $m$ becomes a dummy variable and can be given any value. 	
 		
 		\item[(c)] We also note that the MDI algorithm in \cite{Feng-Zhong} has an additional parameter $r$ which 
 		selects the 1-d quadrature rule. However,  such a choice is not needed here because the underlying QMC rule is used as the 1-d quadrature rule. 
 	\end{itemize}
 \end{remark}

 We are now ready to define our MDI-LR algorithm, which is denoted by  MDI-LR$(d,g, \Omega_d, \mathbf{N}_d, m)$, by using the above MDI algorithm  to evaluate $\widehat{Q}_{n,d}(f)$ in  \eqref{iQMC}.
 
 \begin{algorithm}[H]
{\color{black} 
	\caption{MDI-LR($f$, $\Omega, d, a, n$)} 
	\label{alg:1-1}
		\hspace*{0.02in} {\bf Inputs:} 
		$f, \Omega, d, a, n$. \\
		\hspace*{0.02in} {\bf Output:}
		$\widehat{Q}_{n,d}(f)=Q_{n+n^*,d}(f)$.
	\begin{algorithmic}[1]
		\State Initialize $\mathbf{z}=(1,a,a^2,\cdots,a^{d-1})$, $J=0, Q=0$,  $m=1$. \\
		 Construct matrix $A$ and $\mathbf{b}$ by (\eqref{eq3.1}).
		\State $g(\mathbf{y}) :=f(A^{-1}(\mathbf{y}-\mathbf{b}))$. 
		 \State Generate the vector $\mathbf{N}_d$ by \eqref{eq3.5}. 
		\State $\widehat{\Omega}:=\bigl\{\mathbf{y}\, |  \mathbf{y}=A\mathbf{x}+\mathbf{b},\, \mathbf{x}\in \Omega \bigr\}.$
		\State $J=$MDI$(d,g, \widehat{\Omega}, \mathbf{N}_d, m)$  
		\State $Q=\frac{J}{|A|}$. \\
		\Return $\widehat{Q}_{n,d}(f)=Q$.
	\end{algorithmic}
}
\end{algorithm}

Noting that here we set $m=1$, that is, the dimension is reduced by $1$ at each dimension iteration, this is because the numerical tests of  \cite{Feng-Zhong} shows that when $m=1$ the MDI algorithm is more efficient than when $m>1$. Also, the upper limit vector $\mathbf{N}_d$ depends on the choice of the underlying QMC rule. In Lemma \ref{lem-3.2} we showed that when $N=[n^{\frac{1}{d}}]$ and $a=N$, then $n_1=n_2=\cdots=n_d=N$, that is, the number of integration points is the same in each (transformed) coordinate direction. 
}

\section{Numerical performance tests} \label{sec-4}
In this section, we present extensive and purposely designed numerical experiments to gauge the performance of the proposed 
MDI-LR algorithm and to demonstrate its superiority over the standard implementations of the QMC lattice rule (SLR) and the improved lattice rule (Imp-LR) for computing high dimensional integrals. All our numerical experiments are done in Matlab on a desktop PC with Intel(R) Xeon(R) Gold 6226R CPU 2.90GHz and 32GB RAM.

\subsection{Two and three-dimensional tests} \label{sec-4.1}
We first test our MDI-LR on simple 2- and 3-d examples and to compare its performance 
(in terms of the CPU time) with the SLR and Imp-LR methods .

\medskip
{\bf Test 1.} Let $\Omega=[0,1]^2$ and consider the following 2-d integrands:
\begin{equation}\label{Ex1}
	f(x):= \frac{x_2\exp\bigl(x_1x_2\bigr)}{e-2}; \qquad \widehat{f}(x): = \sin\bigl(2\pi+x_1^2+x_2^2\bigr).
\end{equation}

Table \ref{tab:1-1} and \ref{tab:1-2} present the computational results (errors and CPU times) of
the SLR, Imp-LR and MDI-LR method for approximating $I_2(f)$ and $I_2(\widehat{f})$, respectively. Recall that the Imp-LR is obtained by adding some integration points on the boundary of the domain in the transformed  coordinates, and the MDI-LR algorithm provides a fast implementation of  the Imp-LR using the MDI approach. From Table \ref{tab:1-1} and \ref{tab:1-2}, we observe that all three methods require very little CPU time. The difference is almost negligible although the SLR is faster than the other two methods. Moreover, the Imp-LR and MDI-LR methods use function values at some additional sampling points on the boundary, which leads to higher accuracy compared to the SLR method as we predicated earlier. 
\begin{table}[H]
	\centering
	\begin{tabular}{ccccccc}
		\cline{1-7} \noalign{\smallskip}
		\multicolumn{1}{c}{}&\multicolumn{2}{c}{SLR(Standard LR)}&\multicolumn{2}{c}{Imp-LR(Improved LR)} &\multicolumn{2}{c}{MDI-LR} \\
		\cline{2-7} \noalign{\smallskip}
		\makecell[c]{Total \\ nodes ($n$)}&\makecell[c]{Relative\\ error} &\makecell[c]{CPU\\time} &\makecell[c]{Relative\\ error} &\makecell[c]{CPU\\time}&\makecell[c]{Relative\\ error} &\makecell[c]{CPU\\time} \\
		\noalign{\smallskip}\hline\noalign{\smallskip}
		101   & $1.332\times 10^{-2}$   & 0.0422&$1.218\times 10^{-3}$ & 0.0423 &$1.218\times 10^{-3}$ & 0.0877 \\
		501   & $5.169\times 10^{-3}$  & 0.0567 &$2.520\times 10^{-4}$ &0.0547&$2.520\times 10^{-4}$ &0.3230 \\
		1001   & $4.051\times 10^{-3}$  & 0.0610  &$1.269\times 10^{-4}$ &0.0657&$1.269\times 10^{-4}$ &0.5147\\
		5001  & $2.570\times 10^{-3}$  & 0.0755&$ 2.489\times 10^{-5}$ & 0.0754 &$ 2.489\times 10^{-5}$ &1.6242\\
		10001     & $2.094\times 10^{-4}$ & 0.0922 &$ 1.220\times 10^{-5}$    & 0.0921 &$ 1.220\times 10^{-5}$    & 3.9471 \\
		40001    & $7.294\times 10^{-5}$&  0.1782&$3.050\times 10^{-6}$& 0.1787&$3.050\times 10^{-6}$ & 7.0408  \\
		
		\noalign{\smallskip}\hline
	\end{tabular}
	\caption{Relative errors and CPU times of SLR, Improved LR and MDI-LR simulations with  $N=[n^{\frac{1}{d}}], a=N$,  $n_1=n_2=N$ for approximating $I_2(f)$.}
	\label{tab:1-1}       
\end{table} 
\begin{table}[H]
	\centering
	\begin{tabular}{ccccccc}
		\cline{1-7} \noalign{\smallskip}
		\multicolumn{1}{c}{}&\multicolumn{2}{c}{SLR(Standard LR)}&\multicolumn{2}{c}{Imp-LR(Improved LR)} &\multicolumn{2}{c}{MDI-LR} \\
		\cline{2-7} \noalign{\smallskip}
		\makecell[c]{Total \\ nodes ($n$)}&\makecell[c]{Relative\\ error} &\makecell[c]{CPU\\time} &\makecell[c]{Relative\\ error} &\makecell[c]{CPU\\time}&\makecell[c]{Relative\\ error} &\makecell[c]{CPU\\time} \\
		\noalign{\smallskip}\hline\noalign{\smallskip}
		101   & $1.163\times 10^{-2}$   & 0.0415&$1.072\times 10^{-3}$ & 0.0410 &$1.072\times 10^{-3}$ & 0.0980 \\
		501   & $6.794\times 10^{-3}$  &0.0539 &$1.399\times 10^{-4}$ & 0.0546 &$1.399\times 10^{-4}$ & 0.3498 \\
		1001   & $3.814\times 10^{-3}$  &0.0647 &$
		7.040\times 10^{-5}$ &0.0653  &$7.040\times 10^{-5}$ &0.5028 \\
		5001  & $1.858\times 10^{-3}$  &0.0723 &$1.3411\times 10^{-5}$ &0.0733 &$1.341\times 10^{-5}$ &1.7212 \\
		10001    & $1.175\times 10^{-4}$ & 0.0965  &$6.759\times 10^{-6}$    & 0.0945  &$6.759\times 10^{-6}$    & 3.4528  \\
		40001    & $2.937\times 10^{-5}$& 0.1386&$ 1.689\times 10^{-6}$    &  0.1399&$ 1.689\times 10^{-6}$    &  6.1104 \\
		\noalign{\smallskip}\hline
	\end{tabular}
	\caption{Relative errors and CPU times of SLR, Improved LR and MDI-LR simulations with   $N=[n^{\frac1{d}}], a=N$,  $n_1=n_2=N$ or approximating $I_2(\widehat{f})$.}
	\label{tab:1-2}       
\end{table}
\medskip
{\bf Test 2.}  Let $\Omega=[0,1]^3$ and we consider the following 3-d integrands:
\begin{equation}\label{ex3}
	f(x):= \frac{\exp\bigl(x_1+x_2+x_3\bigr)}{(e-1)^3};\qquad \widehat{f}(x):=\sin\bigl(2\pi+x_1^2+x_2^2+x_3^2\bigr).
\end{equation}

Tables \ref{tab:1-3} and \ref{tab:1-4} present the simulation results (errors and CPU time) of the SLR, Imp-LR, and MDI-LR methods for computing  $I_3(f)$ and $I_3(\widehat{f})$ in \textbf{Test 2.} We observe that the SLR method requires less CPU time in both simulations. The advantage of the MDI-LR method in accelerating the computation does not materialize in low dimensions as seen in {\bf Test 1}. Once again, the Imp-LR and MDI-LR  have higher accuracy compared to the SLR method because they use additional sampling points on the boundary of the transformed domain.
\begin{table}[H]
	\centering
	\begin{tabular}{ccccccc}
		\cline{1-7} \noalign{\smallskip}
		\multicolumn{1}{c}{}&\multicolumn{2}{c}{SLR(Standard LR)}&\multicolumn{2}{c}{Imp-LR(Improved LR)} &\multicolumn{2}{c}{MDI-LR} \\
		\cline{2-7} \noalign{\smallskip}
		\makecell[c]{Total \\ nodes ($n$)}&\makecell[c]{Relative\\ error} &\makecell[c]{CPU\\time} &\makecell[c]{Relative\\ error} &\makecell[c]{CPU\\time}&\makecell[c]{Relative\\ error} &\makecell[c]{CPU\\time} \\
		\noalign{\smallskip}\hline\noalign{\smallskip}
		101   & $3.426\times 10^{-3}$   & 0.0574&$4.985\times 10^{-3}$ &  0.0588 &$4.985\times 10^{-3}$ & 0.0877 \\
		1001   & $6.276\times 10^{-3}$  & 0.0634 &$1.249\times 10^{-3}$ & 0.0654 &$1.249\times 10^{-3}$ &0.2684 \\
		10001     & $9.920\times 10^{-4}$ & 0.0833 &$3.124\times 10^{-4}$    & 0.0877  &$3.124\times 10^{-4}$    & 0.6322  \\
		100001    & $ 5.717\times 10^{-4}$& 0.1500&$5.907\times 10^{-5}$    & 0.1499&$5.907\times 10^{-5}$    & 2.5866  \\
		1000001   & $1.369\times 10^{-5}$ &  1.0589&$ 1.249\times 10^{-5}$ & 1.0587 &$1.249\times 10^{-5}$ & 14.737\\
		
		10000001  & $8.441\times 10^{-6}$ &   9.8969  & $ 3.124\times 10^{-6}$ &10.280 & $3.124\times 10^{-6}$ &91.897 \\
		\noalign{\smallskip}\hline
	\end{tabular}
	\caption{Relative errors and CPU times of SLR, Improved LR and MDI-LR simulations with   $N=[n^{\frac1{d}}],a=N$, $n_1=n_2=n_3=N$  for computing  $I_3(f)$.}
	\label{tab:1-3}       
\end{table}
\begin{table}[H]
	\centering
	\begin{tabular}{ccccccc}
		\cline{1-7} \noalign{\smallskip}
		\multicolumn{1}{c}{}&\multicolumn{2}{c}{SLR(Standard LR)}&\multicolumn{2}{c}{Imp-LR(Improved LR)} &\multicolumn{2}{c}{MDI-LR} \\
		\cline{2-7} \noalign{\smallskip}
		\makecell[c]{Total \\ nodes ($n$)}&\makecell[c]{Relative\\ error} &\makecell[c]{CPU\\time} &\makecell[c]{Relative\\ error} &\makecell[c]{CPU\\time}&\makecell[c]{Relative\\ error} &\makecell[c]{CPU\\time} \\
		\noalign{\smallskip}\hline\noalign{\smallskip}
		101   & $1.866\times 10^{-2}$   &  0.0580&$1.008\times 10^{-3}$ &  0.0554 &$1.008\times 10^{-3}$ &  0.1366 \\
		1001   & $9.746\times 10^{-3}$  & 0.0628&$2.739\times 10^{-4}$ &0.0649 &$2.739\times 10^{-4}$ &0.3804 \\
		10001   & $1.001\times 10^{-3}$ & 0.0820 &$6.337\times 10^{-5}$    & 0.0828 &$6.337\times 10^{-5}$    & 1.1032  \\
		100001  & $7.063\times 10^{-4}$&  0.1443&$1.326\times 10^{-5}$    & 0.1557 &$1.326\times 10^{-5}$    & 4.8794  \\
		1000001 & $2.211\times 10^{-5}$ &   1.1163 &$2.810\times 10^{-6}$ & 1.2104&$2.810\times 10^{-6}$ & 20.305\\
		
		10000001 & $ 1.650\times 10^{-5}$ &   10.207 & $7.026\times 10^{-7}$ &10.427 & $ 7.026\times 10^{-7}$ &101.22\\
		\noalign{\smallskip}\hline
	\end{tabular}
	\caption{Relative errors and CPU times of SLR, Improved LR and MDI-LR simulations with $N=[n^{\frac1{d}}],a=N$, $n_1=n_2=n_3=N$   for computing  $I_3(\widehat{f})$.}
	\label{tab:1-4}       
\end{table}
\subsection{High-dimensional tests}\label{sec-4.2}
Since the MDI-LR method is designed for computing high-dimensional integrals, its performance for $d >> 1$ is more important and anticipated, which is indeed the main task of this subsection. 
First, we test and compare the performance (in terms of CPU time) of the SLR, Imp-LR, and MDI-LR methods for computing high-dimensional integrals as the number of lattice points grows due to the dimension increases.  Then, we also test the performance of the SLR and MDI-LR methods for computing high-dimensional integrals when the  number of lattice points increases slowly in the dimension $d$.

\medskip
{\bf Test 3.}  Let $\Omega=[0,1]^d$ for $2\leq d\leq 50$ and consider the following Gaussian integrand:
\begin{equation}\label{ex5}
	f(x)= \frac{1}{\sqrt{2\pi}}\exp\Bigl(-\frac{1}{2}|x|^2\Bigr),
\end{equation}
where $|x|$ stands for the Euclidean norm of the vector $x\in \mathbb{R}^d$. 

Table \ref{tab:2-1} shows the relative errors and CPU times of SLR, Imp-LR, and MDI-LR methods for approximating the Gaussian integral $I_d(f)$.  The simulation results indicate that SLR and Imp-LR methods are more efficient when $d < 7$, but they struggle to compute integrals when $d>11$ as the number of lattice points increases exponentially in the dimension. However, this is not a problem for the MDI-LR method, which can compute this high-dimensional integral easily. Moreover, the MDI-LR method improves the accuracy of the original QMC rule significantly by adding some integration points on the boundary of the transformed domain.
 
 \begin{table}[H]
 	\centering
 	\begin{tabular}{ccccccc}
 		\cline{1-7} \noalign{\smallskip}
 		\multicolumn{1}{c}{}&\multicolumn{2}{c}{\makecell[c]{SLR(Standard LR)\\ Total Nodes($1+10^d$)}}&\multicolumn{2}{c}{\makecell[c]{Imp-LR(Improved LR)\\ Total Nodes($1.1\times10^d$)} } &\multicolumn{2}{c}{\makecell[c]{MDI-LR\\ Total Nodes($1.1\times10^d$)} } \\
 		\cline{2-7} \noalign{\smallskip}
 		\makecell[c]{Dimension \\ ($d$)}&\makecell[c]{Relative\\ error} &\makecell[c]{CPU\\time} &\makecell[c]{Relative\\ error} &\makecell[c]{CPU\\time}&\makecell[c]{Relative\\ error} &\makecell[c]{CPU\\time} \\
 		\noalign{\smallskip}\hline\noalign{\smallskip}
 		2  & $4.802\times 10^{-3}$ &0.0622 & $5.398\times 10^{-4}$ & 0.0637& $5.398\times 10^{-4}$&0.1335 \\
 		4  & $3.796\times 10^{-3}$ &0.1068 & $1.131\times 10^{-3}$ & 0.1206& $1.131\times 10^{-3}$&0.5780 \\
 		6 & $7.780\times 10^{-3}$ &1.2450 & $1.723\times 10^{-3}$ & 1.2745& $1.723\times 10^{-3}$&1.2890 \\
 		
 		8  & $1.189\times 10^{-2}$ &124.91 & $2.315\times 10^{-3}$ & 126.85& $2.315\times 10^{-3}$&1.4083 \\
 		10  & $1.602\times 10^{-2}$ &13084  & $2.908\times 10^{-3}$ & 13255& $2.908\times 10^{-3}$&3.1418 \\
 		11   & $1.809\times 10^{-2}$ &132927 & $3.204\times 10^{-3}$ & 141665& $3.204\times 10^{-3}$&3.8265 \\
 		12   & failed &failed &failed &failed& $3.501\times 10^{-3}$&4.5919 \\
 		
 		\noalign{\smallskip}\hline
 	\end{tabular}
 	\caption{Relative errors and CPU times of SLR, Improved LR and MDI-LR simulations with  $N=[n^{\frac{1}{d}}],a=N$,   $n_1=\cdots=n_d=N$   for computing $I_d(f)$. }	\label{tab:2-1}     
 \end{table}

\begin{table}[H]
	\centering
	\begin{tabular}{ccccccc}
		\cline{1-7} \noalign{\smallskip}
		\multicolumn{3}{c}{}&\multicolumn{2}{c}{SLR} &\multicolumn{2}{c}{MDI-LR} \\
		\cline{4-7} \noalign{\smallskip}
		\makecell[c]{Dimension \\  ($d$)}&\makecell[c]{Total \\ nodes ($n$)}&\makecell[c]{ $a$ value}&\makecell[c]{Relative\\ error} &\makecell[c]{CPU\\time} &\makecell[c]{Relative\\ error} &\makecell[c]{CPU\\time} \\
		\noalign{\smallskip}\hline\noalign{\smallskip}
		2 & 1+$10^{3}$ & 31  & $4.8020\times 10^{-4}$ &0.0369 & $6.1474\times 10^{-5}$&0.432905 \\
		
		6 & 1+$10^{6}$ & 10  & $7.7798\times 10^{-3}$ &1.2450 & $1.7745\times 10^{-3}$&0.790102 \\
		10 & 1+$ 10^{6}$ & 4  & $5.3673\times 10^{-2}$ &1.2453 & $1.8683\times 10^{-2}$&0.582487 \\
		14 & 1+$ 10^{8}$ & 4  & $7.9282\times 10^{-2}$ &144.759 & $2.6253\times 10^{-2}$&0.536131 \\  
		
		18 & 1+$ 10^{9}$ & 3  & $1.5827\times 10^{-1}$ &1649.59 & $6.1158\times 10^{-2}$&0.774606 \\
		22 & 1+$ 10^{10}$ & 3  & $2.0007\times 10^{-1}$ &18694.04 & $7.5249\times 10^{-2}$&0.702708 \\  
		
		26 & 1+$ 10^{11}$ & 3  & $2.4341\times 10^{-1}$ &217381.41 & $8.9527\times 10^{-2}$&0.866122 \\
		30 & 1+$ 10^{11}$ & 3  & $2.9009\times 10^{-1}$ &269850.87 & $1.0399\times 10^{-1}$&1.045107 \\
		
		\noalign{\smallskip}\hline
	\end{tabular}
	\caption{Relative errors and CPU times of SLR and MDI-LR simulations with the same number of integration points for computing $I_d(f)$. }	\label{tab:2-2}     
\end{table}
Table \ref{tab:2-2} shows the relative errors and CPU times of the  SLR and MDI-LR methods for computing  $I_d(f)$  {\color{black} when the number of lattice points  increase slowly in the dimension $d$.} As the dimension increases, the CPU time required by the SLR method also increases sharply (see Figure \ref{fig.6}). When approximating the Gaussian integral of about 30 dimensions with $10^{11}$ lattice points, the SLR method requires $74$ hours to obtain a result with relatively low accuracy. In contrast, the MDI-LR method only takes about one second to obtain a more accurate value, this demonstrates that  the acceleration effect of the MDI-LR method is quite dramatic.

\begin{figure}[H]
	\centerline{
		\includegraphics[width=2.5in,height=2in]{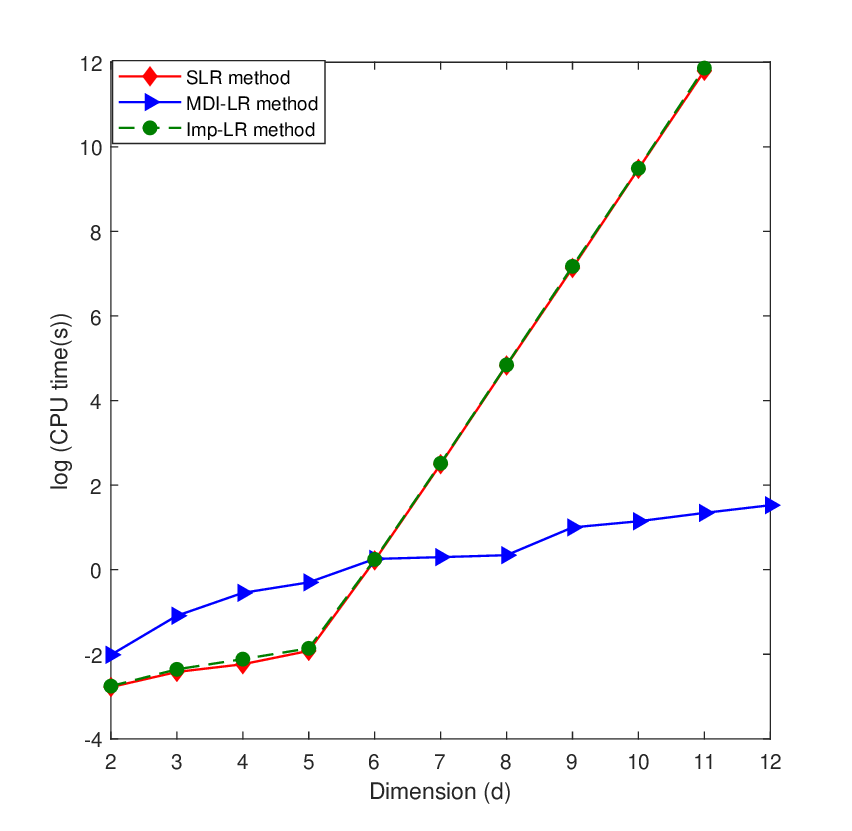}
		\includegraphics[width=2.5in,height=2in]{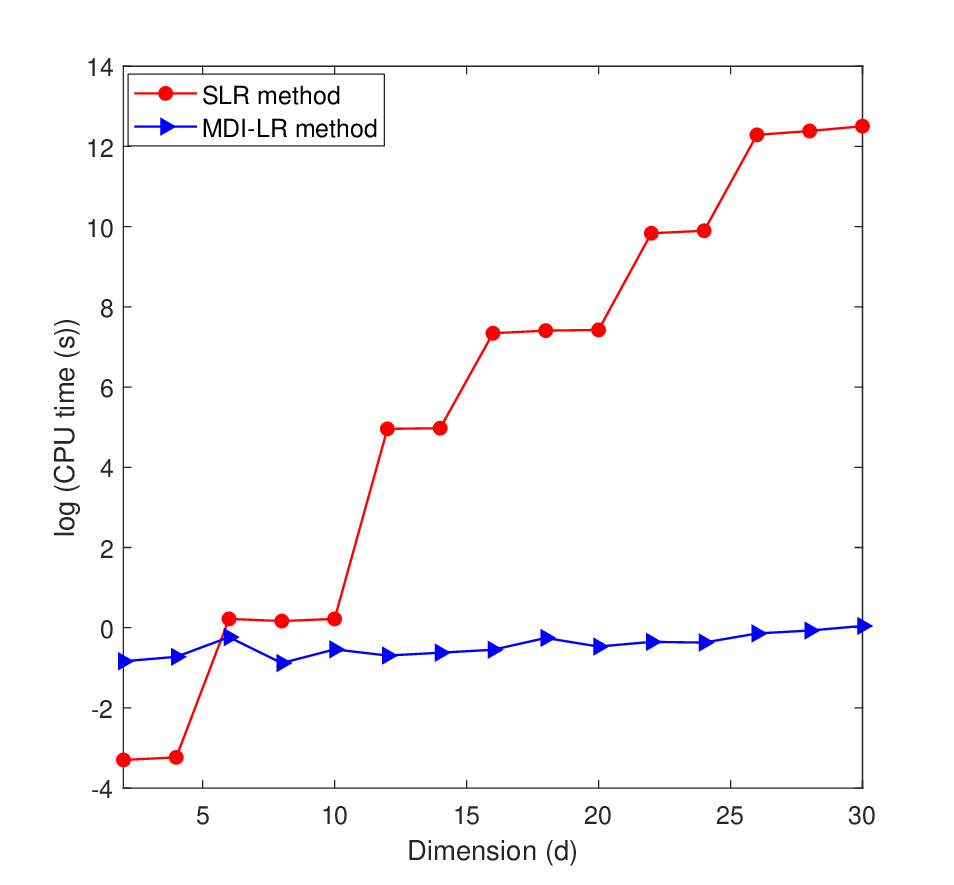}
	}
	\caption{CUP time comparison of SLR and MDI-LR simulations:  the number of lattice points increases in dimension (left),
		{\color{black} the number of lattice points increases slowly} (right). }	\label{fig.6}
\end{figure}

%
It is well known that it is difficult to obtain high-accuracy approximations in high dimensions because the number of integration points required is enormous. 
A natural question is whether the MDI-LR method can handle very high (i.e., $d\approx 1000$) dimensional integration with reasonable accuracy. First, we note that the answer is machine dependent, as expected. Next, we present a test on the computer at our disposal  to provide a positive answer to this question

\medskip
{\bf Test 4.} Let $\Omega=[0,1]^d$ and consider the following integrands:

\begin{equation}\label{ex-5}
	f(x)= \exp\Bigl(\sum_{i=1}^{d}(-1)^{i+1}x_i \Bigr), 
	\qquad \widehat{f}(x)= \prod_{i=0}^{d} \frac{1}{0.9^2+(x_i-0.6)^2}.
\end{equation}
We use the algorithm MDI-LR  to compute $I_d(f)$ and $I_d(\widehat{f})$ with parameters $a = 8, 10$, and an increasing sequence of $d$. The computed  results  are presented in Table \ref{tab:2-3}.  The simulation is stopped at $d = 1000$ because it is already in the very high dimension regime. These tests demonstrate the efficacy and potential of the MDI-LR method in efficiently computing high-dimensional integrals. 
However,  we note that in terms of efficiency and accuracy, the MDI-LR method underperforms its two companion methods,  namely, MDI-TP \cite{Feng-Zhong} and MDI-SG \cite{Z-F} methods. The main reason for the underperformance is that the original lattice rule is unable to provide high-accuracy integral approximations and the MDI-LR is only a fast implementation algorithm (i.e., solver) for the original lattice rule. 
Nevertheless, the lattice rule has its own advantages, such as allowing flexible integration points and giving better results for periodic integrands. 

\begin{table}[H]
	\centering
	\begin{tabular}{ccccccc}
		\cline{1-7} \noalign{\smallskip}
		\multicolumn{1}{c}{}&\multicolumn{3}{c}{\makecell[c]{$I_d(f)$\\ Nodes($1\times8^d$)}} &\multicolumn{3}{c}{\makecell[c]{$I_d(\widehat{f})$\\ Nodes($1\times20^d$)} } \\
		\cline{2-7} \noalign{\smallskip}
		\makecell[c]{Dimension\\($d$)}&\makecell[c]{$a$ value} &\makecell[c]{Relative\\ error} &\makecell[c]{CPU\\time(s)}&\makecell[c]{$a$ value} &\makecell[c]{Relative\\ error} &\makecell[c]{CPU\\time(s)} \\
		\noalign{\smallskip}\hline\noalign{\smallskip}
		10    &8 & $6.4884\times 10^{-3}$ &0.4329063 &20& $1.6107\times 10^{-3}$ &0.9851172 \\  
		100  &8  & $6.3022\times 10^{-2}$ &71.253076&20  & $1.6225\times 10^{-2}$ &11.1203255\\ 
		300  &8     & $1.7740\times 10^{-1}$ &1856.91018 &20 & $4.9469\times 10^{-2}$ &37.0903112  \\
		500  &8   & $2.7781\times 10^{-1}$ &8076.92429  &20 & $8.3801\times 10^{-2}$ &65.9497657     \\
		700 &8     & $3.6597\times 10^{-1}$ &20969.96162 &20  & $1.1925\times 10^{-1}$ &108.989057  \\
		900  &8     & $4.4337\times 10^{-1}$ &47870.50843  &20   & $1.5587\times 10^{-1}$ &157.487672  \\
		1000 &8    & $4.7845\times 10^{-1}$ & 69991.88017&20   & $1.7462\times 10^{-1}$ &189.132615  \\
		\bottomrule
	\end{tabular}
	\caption{Computed results for $I_d(f)$ and $I_d(\widehat{f})$ by  algorithm MDI-LR.}\label{tab:2-3}
\end{table}

\section{Influence of parameters}\label{sec-5}
The original MDI algorithm involves three crucial input parameters: $r$, $m$, and $N$. The parameter $r$ determines the one-dimensional basis value quadrature rule, while $m$ sets the step size in the multidimensional iteration, and $N$ represents the number of integration points in each coordinate direction. The algorithm MDI-LR  is similar to the original MDI, but uses the QMC rank-one lattice rule with generating vector $\mathbf{z}$, so the parameter $r$ is muted.  Here we focus on the Korobov approach in constructing the generating  vector $\mathbf{z}$, which is defined as $\mathbf{z}=\mathbf{z} (a):=(1,a,a^2, \cdots,a^{d-1})$. 
{\color{black} Moreover, 
the improved tensor product rule (in the transformed coordinate system) implemented by the algorithm Imp-LR has a variable upper limits in the summation
(cf. \eqref{iQMC}), hence,  $N$ is now replaced by $\mathbf{N}_d$ which is determined by the underlying QMC lattice rule.  Furthermore, as explained earlier, we set $m=1$ due to our experience in \cite{Feng-Zhong}. As a result,  the only parameter to select is $a$. 
Below, we first test the influence of the Korobov parameter $a$ on the efficiency of the algorithm MDI-LR  and then test dependence of its performance  on $\mathbf{N}_d$ and $d$. 
}
 
\subsection{Influence of parameter $a$}\label{sec-5.1}
In this subsection, we investigate the impact of the generating vector $\mathbf{z}=\mathbf{z}(a):=(1,a,a^2,\cdots,a^{d-1})$ in the  algorithm MDI-LR . We note that similar methods can be constructed using  other $\mathbf{z}$.

\medskip
{\bf Test 5.} Let $\Omega=[0,1]^d$ and consider the following integrands:
\begin{alignat*}{2}\label{ex6}
	&f(x)=  \frac{1}{\sqrt{2\pi}}\exp\Bigl(-\frac{1}{2}|x|^2\Bigr),
	&&\qquad \widehat{f}(x)= \cos \Bigl(2\pi+2\sum_{i=1}^{d} x_i \Bigr), \\
	&  \widetilde{f}(x)= \prod_{i=0}^{d} \frac{1}{0.9^2+(x_i-0.6)^2}.
\end{alignat*}
We compare the performance of the algorithm MDI-LR  with different Korobov parameters $a$ while holding other parameters unchanged when computing $I_d(f)$, $I_d(\widehat{f})$, and $I_d(\widetilde{f})$.
\\
\begin{figure}[H]
	\centerline{
		\includegraphics[width=2.5in,height=2in]{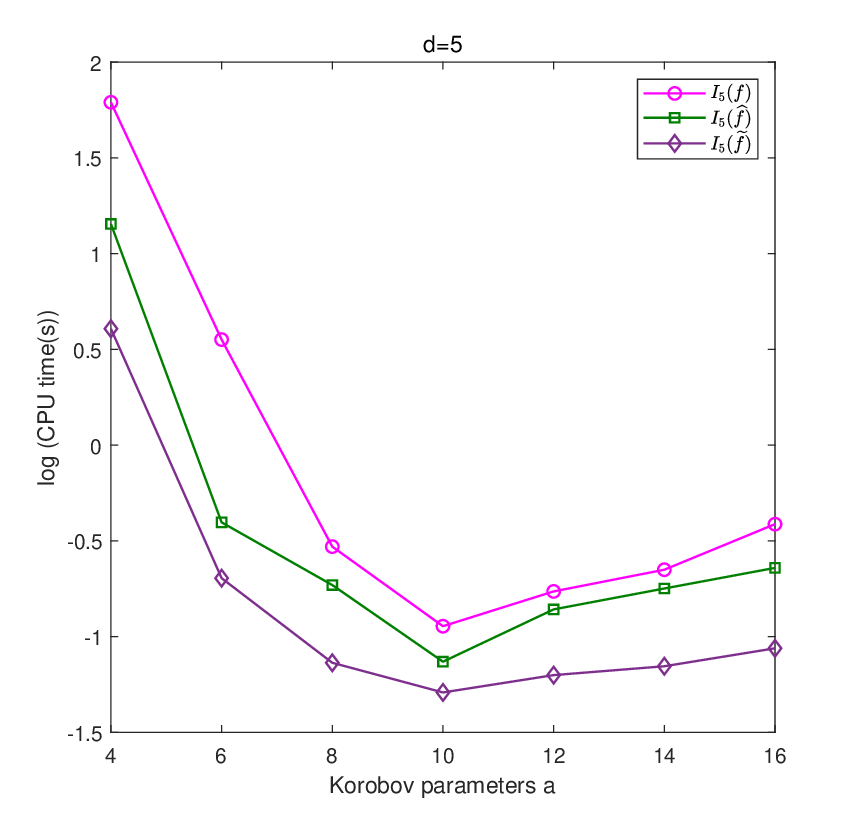}
		\includegraphics[width=2.5in,height=2in]{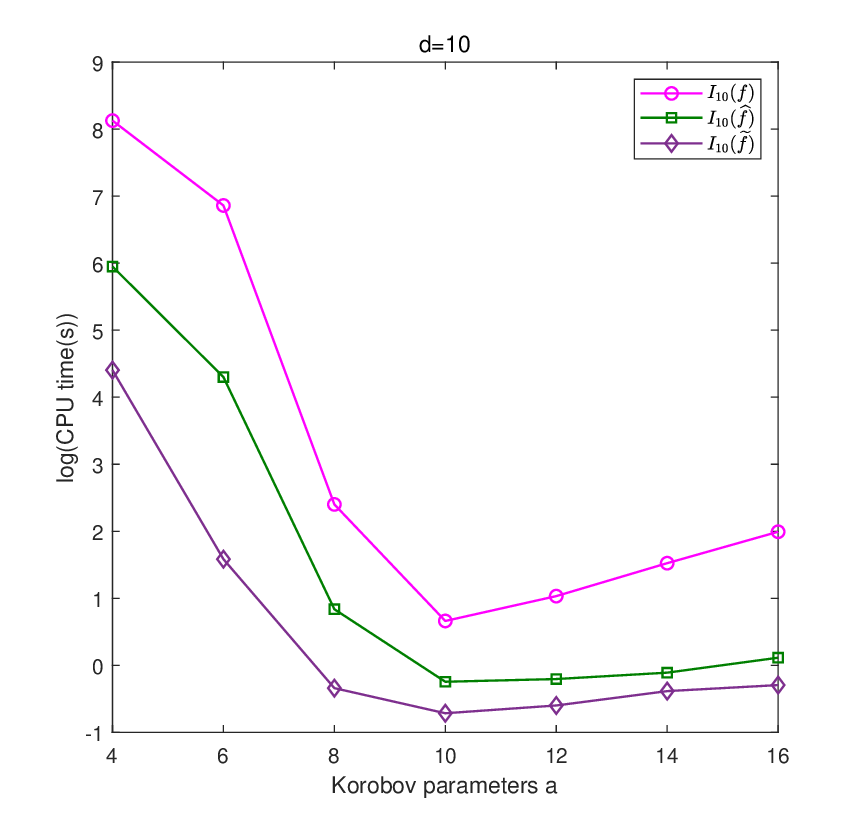}
	}
	\centerline{
	\includegraphics[width=2.5in,height=2in]{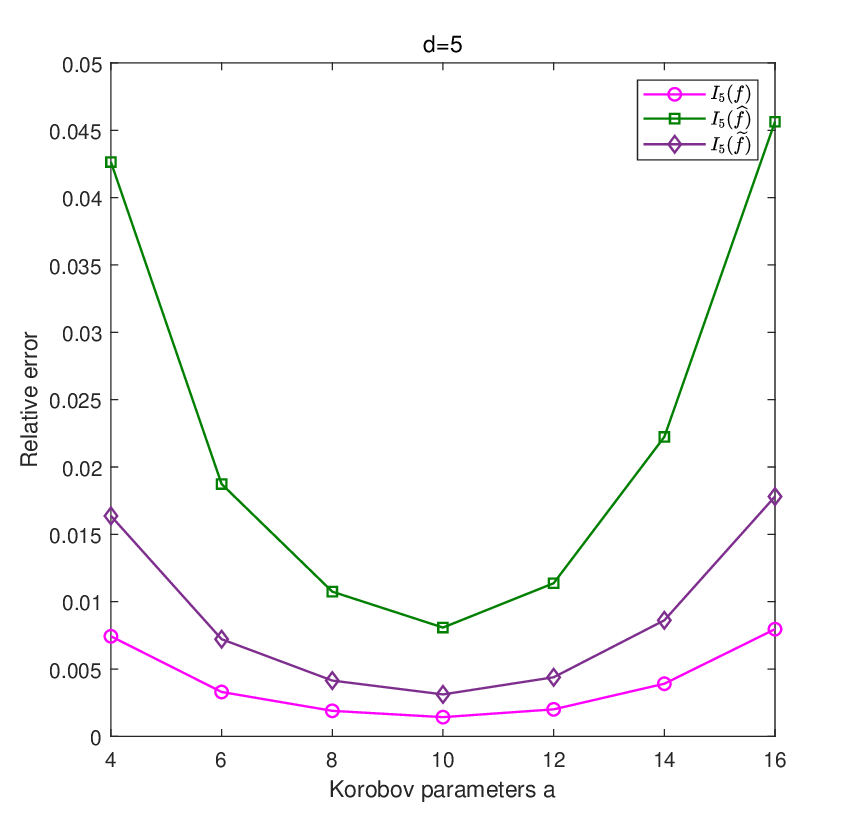}
	\includegraphics[width=2.5in,height=2in]{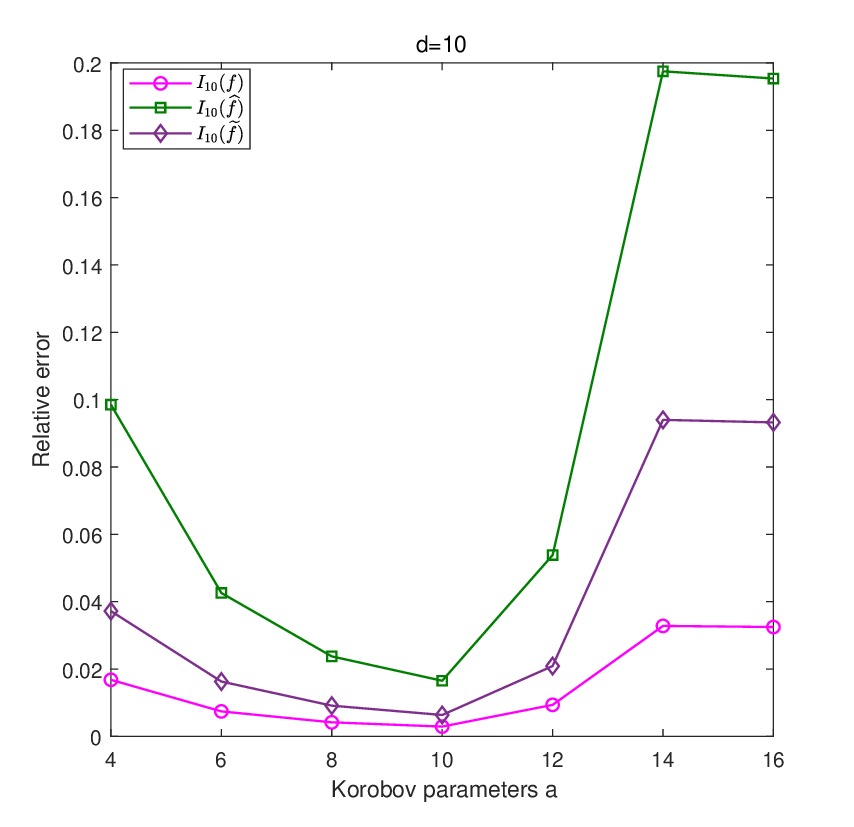}
}
	\caption{Performance comparison of algorithm MDI-LR  with $n=1+10^{d}$ and $a=4,6,8,10,12,14,16$ for computing $I_d(f)$, $I_d(\widehat{f})$ and $I_d(\widetilde{f})$. Top left: $d=5$, CPU time comparison. Top right: $d=10$, CPU time comparison.
		Bottom left: $d=5$, comparison of relative errors. Bottom right: $d=10$, comparison of relative errors}	\label{fig.7}
\end{figure}

Figure \ref{fig.7} shows the computed results for $d = 5, 10$ and $a = 4, 6, 8, 10, 12, 14, 16$, respectively. We  observe that the algorithm MDI-LR  with different parameters $a$ has different accuracy  and the effect could be significant. 
These results indicate that the algorithm is most efficient when {\color{black} $a=N$, where $N=[n^{\frac{1}{d}}]$ and $n$ represents the total number of integral points.}  This is because when a smaller $a$ is used, although fewer integration points need to be evaluated in each coordinate direction in the first $d-1$ dimension iterations, since the total number of integral points $n$ is the same, the amount of computation will increase dramatically. When using a larger $a$,  more integration points need to be used in each coordinate direction in the first $d-1$ dimension  iterations. {\color{black} Only when the integration points are equally distributed to each coordinate  direction, the  efficiency of the algorithm MDI-LR can be optimized.}  A total of $100$ points are shown in Figure \ref{fig.8}.  When $a=2$, only $2$ iterations in the $x_1$-direction are needed, but $50$ iterations in the $x_2$-direction must performed, hence,  a total of $52$ iterations in the two directions are required. On the other hand, when $a=20$, a total of $25$ iterations in the two directions are required.  It is easy to check that  the least total of $20$ iterations 
occurs when $a=10$.  The difference in accuracy is obvious, because the different $a$ leads to different generating vector $\mathbf{z}$, which in turn results in different integration points. 
We note that it was already well studied in the literature on how to choose $a$ to achieve the highest accuracy 
(cf. \cite{ }).
\begin{figure}[H]
	\centerline{
		\includegraphics[width=6.35in,height=1.75in]{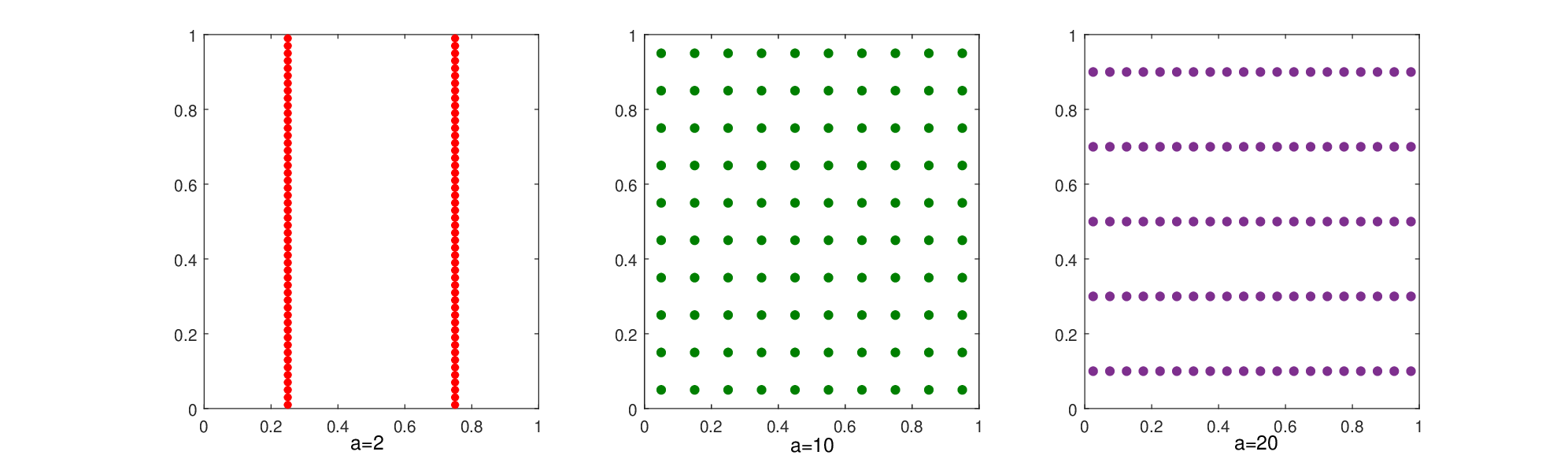}
	}
	\caption{Distribution of 100 integration points in the transformed coordinate system when  $a= 2, 10, 20,$ respectively.}	\label{fig.8}
\end{figure}
\subsection{Influence of parameter $N= [n^{\frac{1}{d}}] $}\label{sec-5.2}
{\color{black}
In the previous section, we know that the algorithm is most efficient when $a=N$, where $N$ represents the number of integration points in each direction. This section aims to investigate the impact of $N$ on the MDI-LR algorithm. For this purpose, we conduct tests by setting $a=N$ and $d=5$ and $d=10$.
}

\smallskip
{\bf Test 6.} Let $\Omega$, $f$, $\widehat{f}$ and $\widetilde{f}$ be the same as in {\bf Test 5.}

\smallskip
Table \ref{tab:3-1}, \ref{tab:3-2}, and \ref{tab:3-3} present a performance comparison for algorithm MDI-LR with  $d=5,10$ and $N=4,6,8,10,12,14,16$, respectively.  We note that the quality of the computed results also depend on  types of the integrands. As expected, more integration points must be used to achieve a good accuracy for very oscillatory and fast growth integrands.
\begin{table}[H]
	\centering
	\begin{tabular}{cccccc}
		\cline{1-6} \noalign{\smallskip}
		\multicolumn{2}{c}{}&\multicolumn{2}{c}{$d=5$ } &\multicolumn{2}{c}{$d=10$ } \\
		\cline{3-6} \noalign{\smallskip}
		\makecell[c]{ N(n)}&\makecell[c]{Korobov\\ parameter $(a)$}&\makecell[c]{Relative\\ error} &\makecell[c]{CPU\\time(s)} &\makecell[c]{Relative\\ error} &\makecell[c]{CPU\\time(s)} \\
		\noalign{\smallskip}\hline\noalign{\smallskip}
		4($1+4^{d}$)  & 4 & $8.6248\times 10^{-3}$   & 0.1456465   & $1.8003\times 10^{-2}$   & 0.3336161\\
		6($1+6^{d}$)  & 6 & $3.8967\times 10^{-3}$   & 0.1911801   & $8.0284\times 10^{-3}$   & 0.5690320\\
		8 ($1+8^{d}$)  &8  & $2.2145\times 10^{-3}$ &  0.3373442   & $4.5314\times 10^{-3}$ &  0.9552591  \\
		10 ($1+10^{d}$)  & 10& $1.4271\times 10^{-3}$& 0.3884146& $2.9078\times 10^{-3}$& 1.9385378   \\	
		12($1+12^{d}$)  &12 &  $9.9601\times 10^{-4}$ & 0.6545521  &  $2.0234\times 10^{-3}$ &3.5639475\\
		14($1+14^{d}$)  &14 &  $ 7.3448\times 10^{-4}$ & 0.7224777   &  $1.4889\times 10^{-3}$ & 6.0036393  \\	
		16($1+16^{d}$)  &16 &  $5.6396\times 10^{-4}$ &   1.0909097  &  $1.1414\times 10^{-3}$ &8.4313528 \\		
		\noalign{\smallskip}\hline
	\end{tabular}
	\caption{Performance comparison of algorithm MDI-LR with $d=5,10, a=N$ and $N=4,6,8,10,12,14,16$ for computing $I_d(f)$.}   \label{tab:3-1}  
\end{table}

\begin{table}[H]
	\centering
	\begin{tabular}{cccccc}
		\cline{1-6} \noalign{\smallskip}
		\multicolumn{2}{c}{}&\multicolumn{2}{c}{$d=5$ } &\multicolumn{2}{c}{$d=10$ } \\
		\cline{3-6} \noalign{\smallskip}
		\makecell[c]{ N(n)}&\makecell[c]{Korobov\\ parameter $(a)$}&\makecell[c]{Relative\\ error} &\makecell[c]{CPU\\time(s)} &\makecell[c]{Relative\\ error} &\makecell[c]{CPU\\time(s)} \\
		\noalign{\smallskip}\hline\noalign{\smallskip}
		4($1+4^{d}$)  & 4 & $4.9621\times 10^{-2}$   & 0.1323887  & $1.0585\times 10^{-1}$   & 0.2775234\\
		6($1+6^{d}$)  & 6 & $2.2181\times 10^{-2}$   & 0.1955847   & $4.6141\times 10^{-2}$   & 0.4267706\\
		8 ($1+8^{d}$)  &8  & $1.2558\times 10^{-2}$ &   0.2689113   & $2.5836\times 10^{-2}$ &  0.5697773  \\
		10 ($1+10^{d}$)  & 10& $8.0791\times 10^{-3}$&  0.3227299& $1.6517\times 10^{-2}$& 0.7828456   \\	
		12($1+12^{d}$)  &12 &  $5.6328\times 10^{-3}$ &  0.4056192  &  $1.1470\times 10^{-2}$ &0.9228344\\
		14($1+14^{d}$)  &14 &  $ 4.1513\times 10^{-3}$ & 0.4940739   &  $8.4305\times 10^{-3}$ & 1.0968489  \\	
		16($1+16^{d}$)  &16 &  $3.1863\times 10^{-3}$ &   0.6079693   &  $6.4576\times 10^{-3}$ &1.2933549 \\		
		\noalign{\smallskip}\hline
	\end{tabular}
	\caption{Performance comparison of  algorithm MDI-LR with $d=5,10, a=N$ and $N=4,6,8,10,12,14,16$ for computing $I_d(\widehat{f})$.}   \label{tab:3-2}  
\end{table}

\begin{table}[H]
	\centering
	\begin{tabular}{cccccc}
		\cline{1-6} \noalign{\smallskip}
		\multicolumn{2}{c}{}&\multicolumn{2}{c}{$d=5$ } &\multicolumn{2}{c}{$d=10$ } \\
		\cline{3-6} \noalign{\smallskip}
		\makecell[c]{ N(n)}&\makecell[c]{Korobov\\ parameter $(a)$}&\makecell[c]{Relative\\ error} &\makecell[c]{CPU\\time(s)} &\makecell[c]{Relative\\ error} &\makecell[c]{CPU\\time(s)} \\
		\noalign{\smallskip}\hline\noalign{\smallskip}
		4($1+4^{d}$)  & 4 & $1.9003\times 10^{-2}$   & 0.1254485   & $3.9895\times 10^{-2}$   & 0.2460844\\
		6($1+6^{d}$)  & 6 & $8.5331\times 10^{-3}$   &0.1802281 & $1.7625\times 10^{-2}$   & 0.3613987\\
		8 ($1+8^{d}$)  &8  & $4.8390\times 10^{-3}$ & 0.2114595  & $9.9155\times 10^{-3}$ &  0.4414383  \\
		10 ($1+10^{d}$)  & 10& $3.1153\times 10^{-3}$&  0.2748469& $6.3531\times 10^{-3}$& 0.4892808   \\	
		12($1+12^{d}$)  &12 &  $2.1729\times 10^{-3}$ &  0.3092816 &  $4.4172\times 10^{-3}$ & 0.5859328\\
		14($1+14^{d}$)  &14 &  $ 1.6018\times 10^{-3}$ & 0.3602077  &  $3.2488\times 10^{-3}$ & 0.6783681 \\	
		16($1+16^{d}$)  &16 &  $1.2296\times 10^{-3}$ &  0.4157161   &  $2.4897\times 10^{-3}$ &0.7819849 \\		
		\noalign{\smallskip}\hline
	\end{tabular}
	\caption{Performance comparison of algorithm MDI-LR with $d=5,10, a=N$ and $N=4,6,8,10,12,14,16$ for computing $I_d(\widetilde{f})$.}   \label{tab:3-3}  
\end{table}

\section{Computational complexity}\label{sec-6} 

\subsection{The relationship between the CPU time and $N$} \label{sec-6.1} 
In this subsection, we examine the relationship between CPU time and the parameter  $N=[n^{\frac1{d}}]$ and $a=N$ using a regression technique based on test data.
\begin{table}[H]
	\centering
	\begin{tabular}{llllll}
		\toprule
		\makecell[c]{Integrand\\}&\makecell[c]{$a$} &\makecell[c]{$m$} &\makecell[c]{$d$} &\makecell[c]{Fitting function}  &\makecell[c]{R-square\\}\\
		
		\midrule
		$f(x)$ &N&  1& 5 & $h_1(N) =(0.007569)*N^{1.772}$ &0.9687   \\
		$\widehat{f}(x)$&N&  1& 5  & $h_2(N)=(0.02326)*N^{1.165}$ &0.9920    \\
		$\widetilde{f}(x)$&N&  1& 5  & $h_3(N)=( 0.03592 )*N^{0.8767}$ &0.9946     \\ 
		$f(x)$&N&  1& 10  &$h_4(N) =( 0.002136)*N^{2.992}$  &0.9968    \\
		$\widehat{f}(x)$&N&  1& 10  &$h_5(N) =(0.05679)*N^{1.125}$  &0.9984    \\
		$\widetilde{f}(x)$&N&  1& 10  &$h_6(N)=( 0.07872)*N^{0.8184}$  &0.9901    \\
		\bottomrule
	\end{tabular}
	\caption{Relationship between the CPU time and parameter $N$.} \label{tab:4-1}
\end{table}
 Figures \ref{fig.9} and \ref{fig.10} show CPU time as a function of $N$ obtained by the  least squares regression with the fitted function given in Table \ref{tab:4-1}. All results show that CPU time grows in proportion to $N^3$.
\begin{figure}[H]
	\centerline{
		\includegraphics[width=1.7in,height=1.65in]{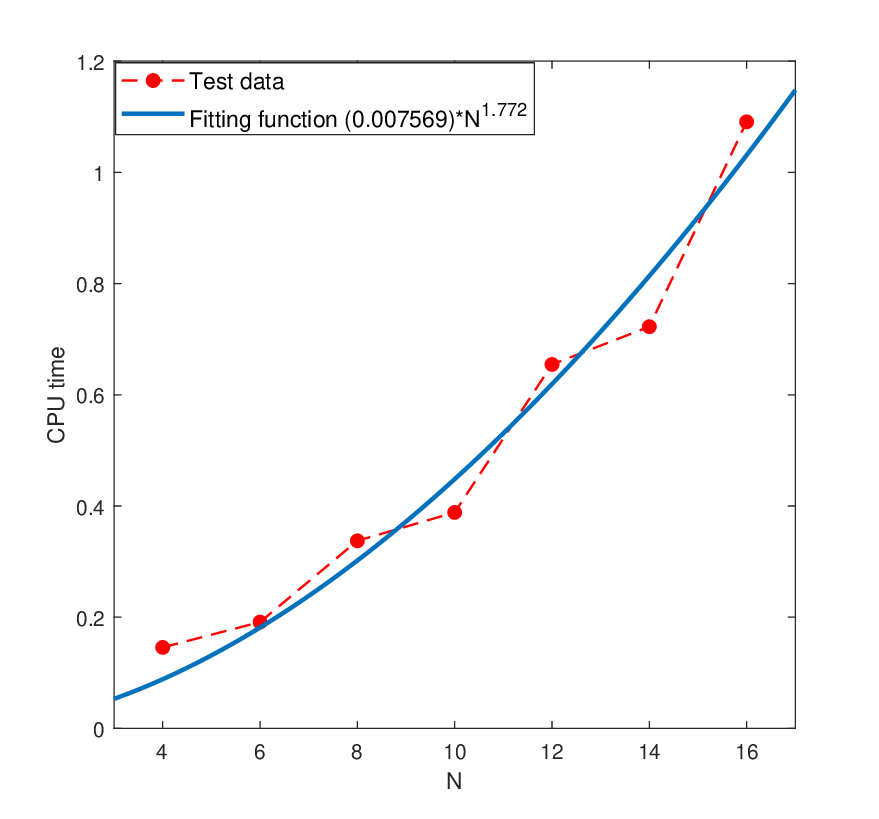}
		\includegraphics[width=1.7in,height=1.65in]{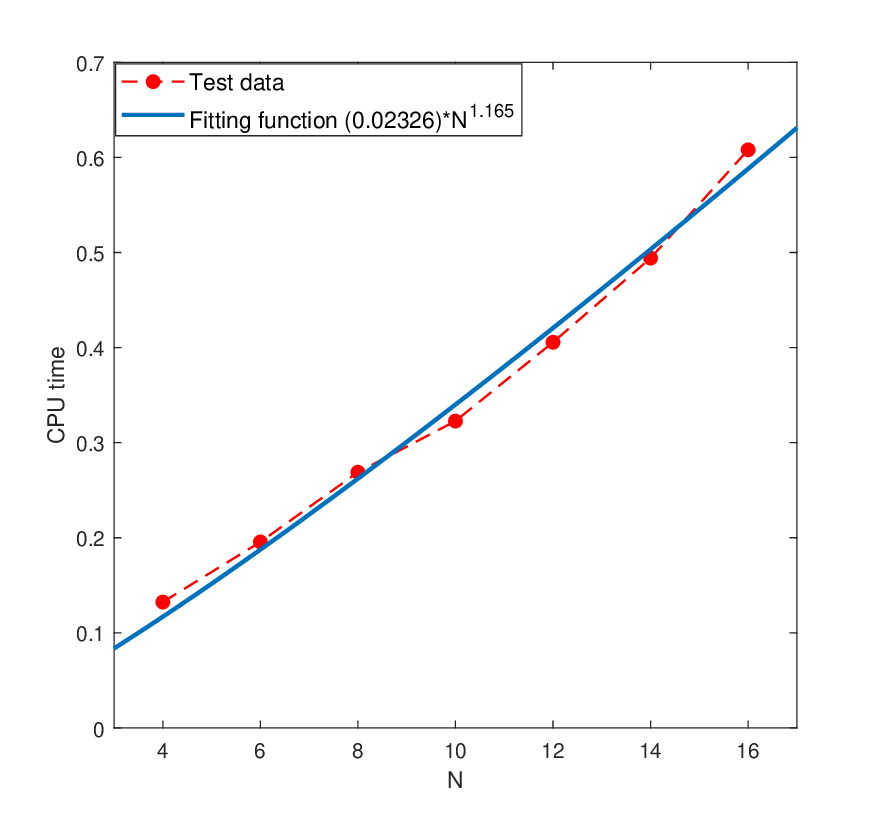}
		\includegraphics[width=1.7in,height=1.65in]{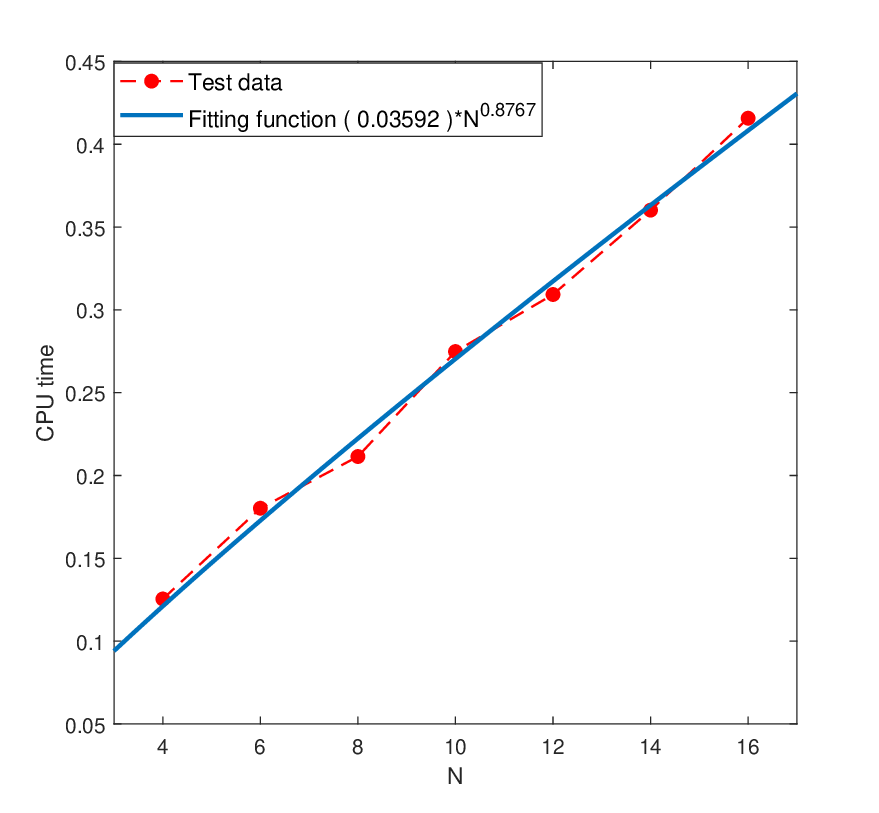}
	}
	\caption{The relationship between the CPU time and parameter $N$ when $d=5$: $I_d(f)$ (left), $I_d(\widehat{f})$ (middle), $I_d(\widetilde{f})$ (right).}
	\label{fig.9}       
\end{figure}

\begin{figure}[H]
	\centerline{
		\includegraphics[width=1.7in,height=1.65in]{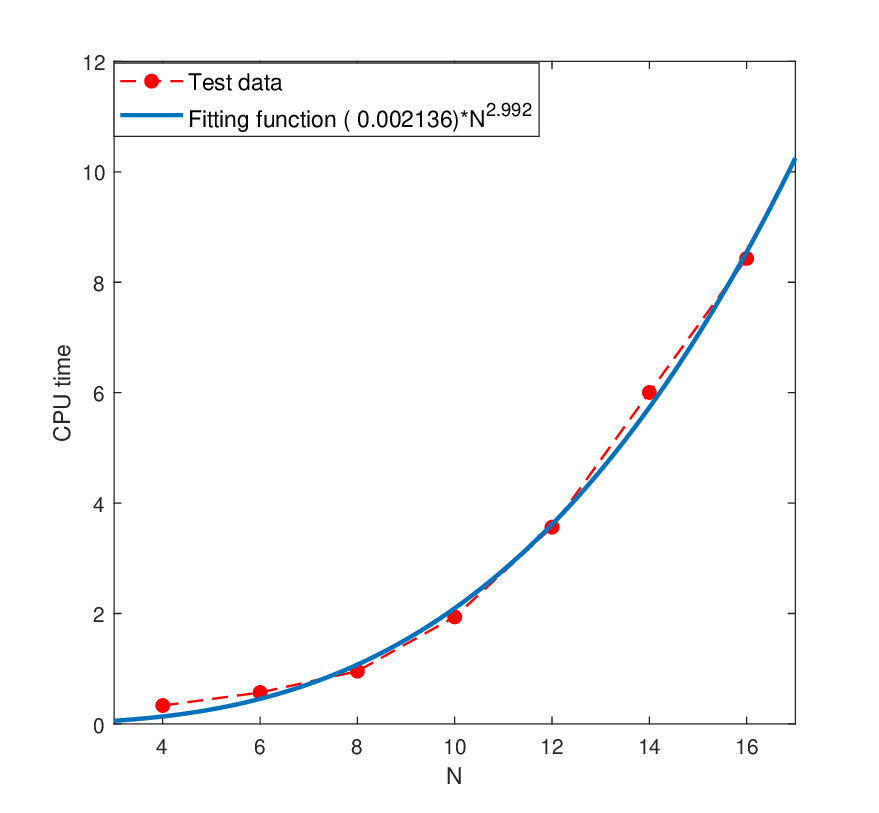}
		\includegraphics[width=1.7in,height=1.65in]{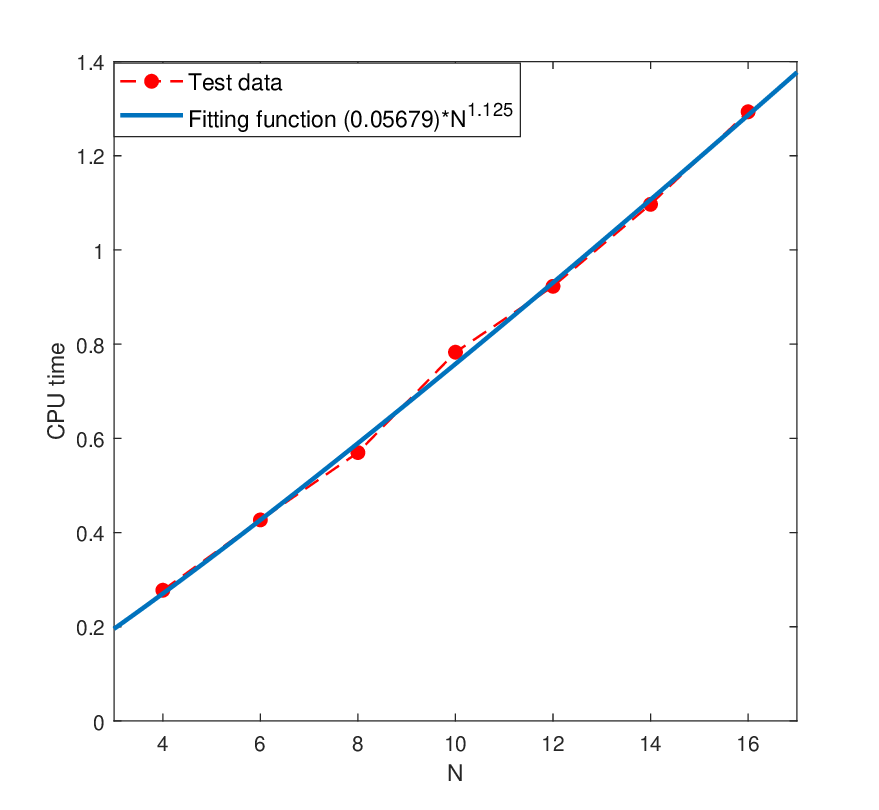}
		\includegraphics[width=1.7in,height=1.65in]{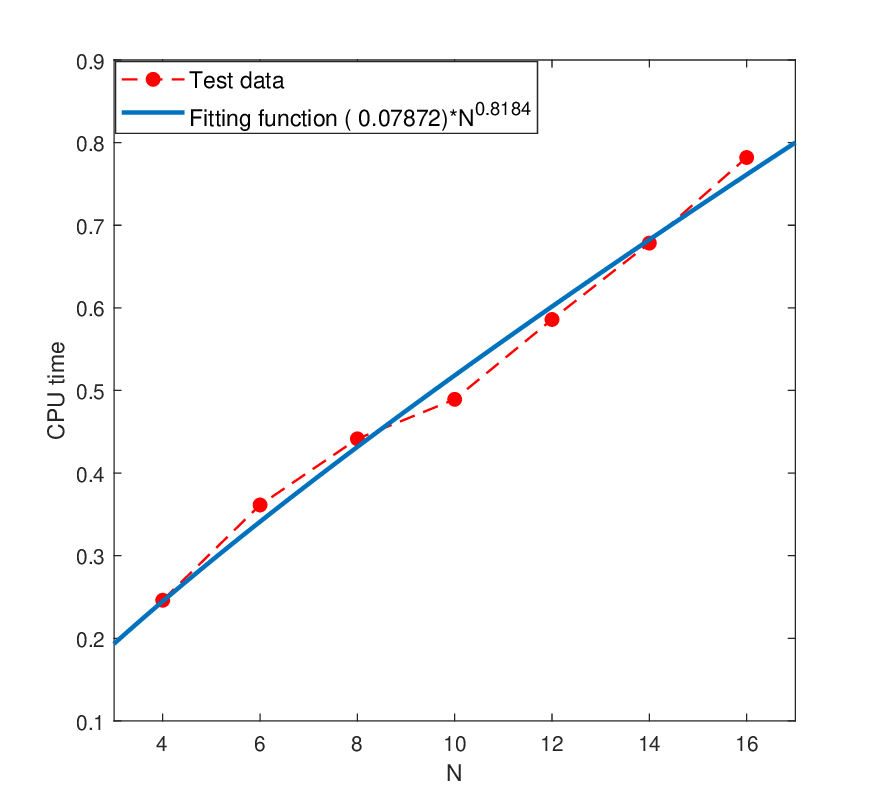}
	}
	\caption{Relationship between the CPU time and parameter $N$ when $d=10$: $I_d(f)$ (left), $I_d(\widehat{f})$ (middle), $I_d(\widetilde{f})$ (right).}
	\label{fig.10}       
\end{figure}

\subsection{The relationship between the CPU time and the dimension $d$}\label{sec-6.2} 
In this subsection, we exploit the computational complexity (in terms of CPU time as a function of $d$) using the least squares regression on numerical test data.

\medskip
{\bf Test 7.} Let $\Omega=[0,1]^d$, we consider the following five integrands:
\begin{alignat*}{2}
	&f_1(x)= \exp\Bigl(\sum_{i=1}^{d}(-1)^{i+1}x_i \Bigr),
	&&\qquad f_2(x)= \prod_{i=1}^{d} \frac{1}{0.9^2+(x_i-0.6)^2},\\
	&f_3(x)= \frac{1}{\sqrt{2\pi}}\exp\Bigl(-\frac{1}{2}|x|^2\Bigr), 
	&&\qquad f_4(x)= \cos\Bigl(2\pi+\sum_{i=1}^{d} 2x_i\Bigr), \\
	&f_5(x)= \exp\Bigl(\sum_{i=1}^{d}(-1)^{i+1}x_i^2 \Bigr),&&\qquad f_6(x)=(1+\sum_{i=1}^{d}x_i)^{-(d+1)} .
\end{alignat*}
Figure \ref{fig:7} displays the the CPU time as functions of $d$ obtained by the least square regression whose analytical expressions are given in Table \ref{tab:4-2}. We note that the parameters of the algorithm  MDI-LR only affect the coefficients of the fitted function, not the power of the polynomials. These results show that the CPU time required by the proposed algorithm MDI-LR  grows at most with polynomial order $O(d^3N^2)$.

\begin{table}[H]
	\centering
	\begin{tabular}{lccclc}
		\cline{1-6}\noalign{\smallskip}
		\makecell[c]{Integrand\\}&\makecell[c]{$a$} &\makecell[c]{$N$}&\makecell[c]{$m$}  &\makecell[c]{Fitting function}  &\makecell[c]{$R$-square\\}\\
		\noalign{\smallskip}\hline\noalign{\smallskip}
		
		$f_1$
		&8&  8&1  & $g_1=(1.057e-06)*N^2d^3$ &0.9973   \\
		&10&  10&1  & $g_2=(1.192e-06)*N^2d^3$ &0.9995   \\
		&20&  20&1& $g_3=( 1.433e-06)*N^2d^{3}$ &0.9978     \\ 
		\hline
		
		$f_2$
		&10&  10&1& $g_4=0.0001774*Nd^{1.611}$ &0.9983     \\
		&14&  14&1& $g_5=0.003028*Nd^{1.147}$ &0.9987     \\
		&20&  20& 1 & $g_6=0.000539*Nd^{1.41}$ &0.9964    \\
		\hline
		
		$f_3$ 
		&8&  8&1&  $g_{7}=(7.334e-06)*N^2d^3$ &0.9983    \\
		&10&  10&1& $g_{8}=(9.321e-06)*N^2d^{3}$ &0.9986     \\
		&14&  14&1& $g_{9}=(1.339e-05)*N^2d^{3}$ &0.9972     \\
		\hline
		
		$f_4$  
		
		&10&  10&1 & $g_{10}=(1.164e-06)*N^2d^3$ &0.9988     \\
		&20&  20&1 & $g_{11}=(1.319e-06)*N^2d^3$ &0.9974     \\
		\hline
		$f_5$ 
		&10&  10&1&$g_{12}=(6.479e-05)*N^2d^{2.557}$ &0.9996       \\
		&14&  14&1&$g_{13}=(1.164e-05)*N^2d^3$ &0.9993       \\
		\hline
		$f_6$ 
		&10&  10&1&$g_{14}=(1.556e-06)*N^2d^3$ &0.9983       \\
		&20&  20&1&$g_{15}=(8.328e-06)*N^2d^{2.431}$ &0.9998       \\
		\noalign{\smallskip}\hline
	\end{tabular}
	\caption{The relationship between CPU time as a function of the dimension $d$.}\label{tab:4-2}     
\end{table}

We assess the quality of the fitted curves using the $R$-square criterion  in Matlab, defined by $R$-$\mbox{\rm square}=1-\frac{\sum_{i}^n(y_i-\widehat{y}_i)^2}{\sum_{i}^n(y_i-\overline{y})^2}$, where $y_i$ is a test data output, $\widehat{y}_i$ is the predicted value, and $\overline{y}$ is the mean of $y_i$. As shown in Table \ref{tab:4-2}, the $R$-square values of all fitted functions are close to $1$, indicating their high accuracy. These results support the observation that the CPU time grows no more than cubically with the dimension $d$. Combined with the results of {\bf Test 6} in Section \ref{sec-6.1}, we conclude that the computational cost of the proposed MDI-LR algorithm scales at most polynomially in the order of  $O(N^2d^3)$.
\begin{figure}[H]
	\centerline{
		\includegraphics[width=1.75in,height=1.65in]{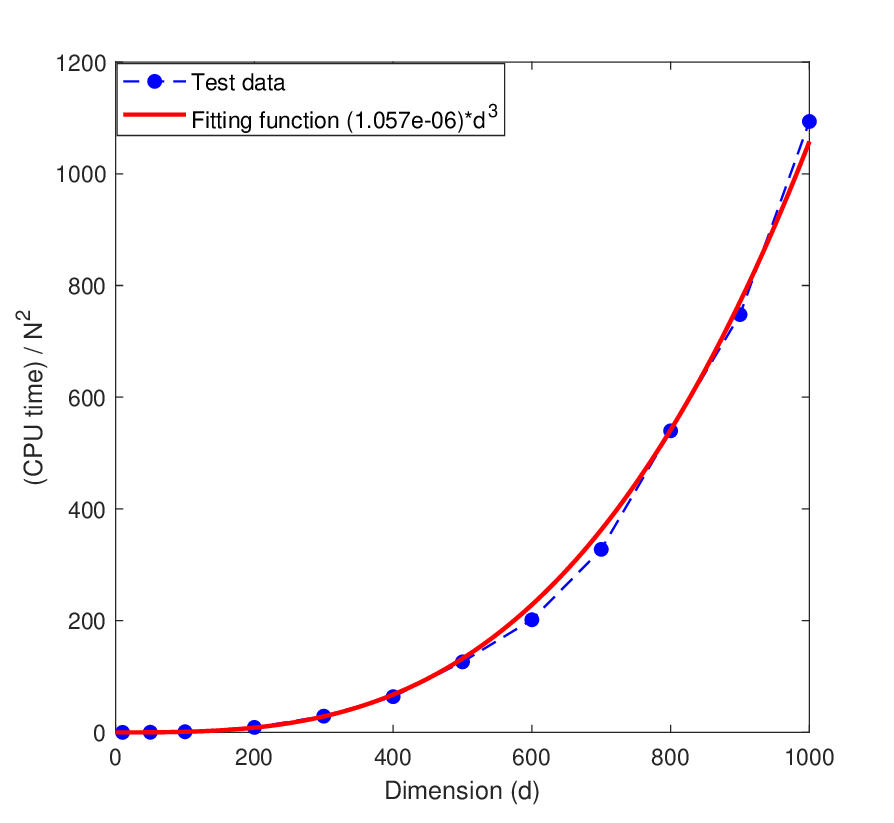}
		\includegraphics[width=1.75in,height=1.65in]{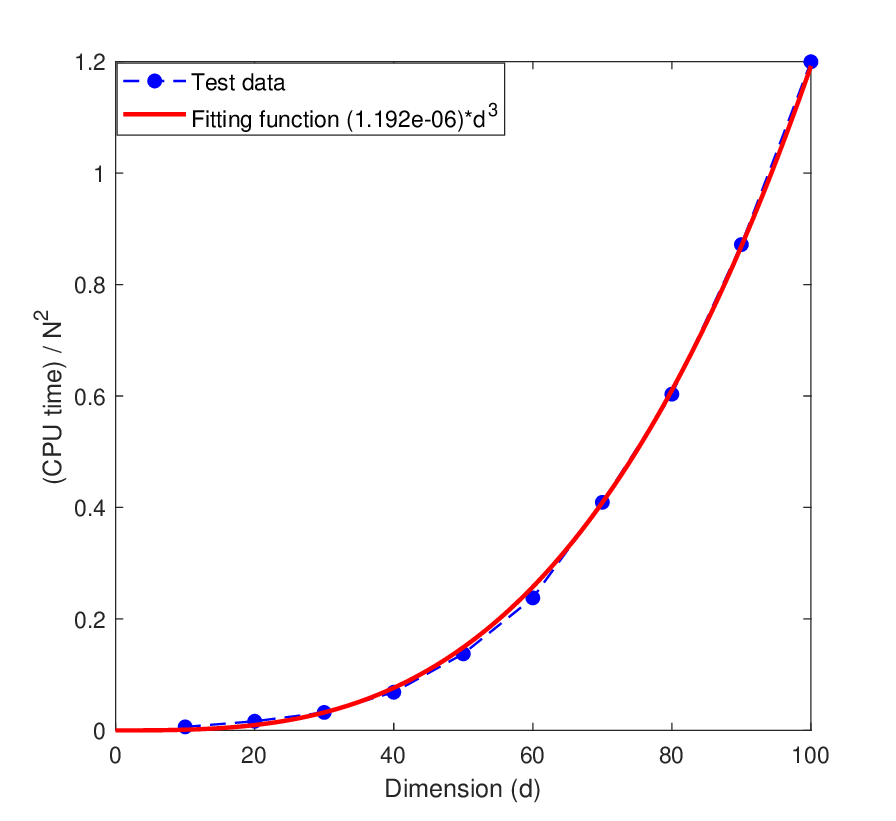}
		\includegraphics[width=1.75in,height=1.65in]{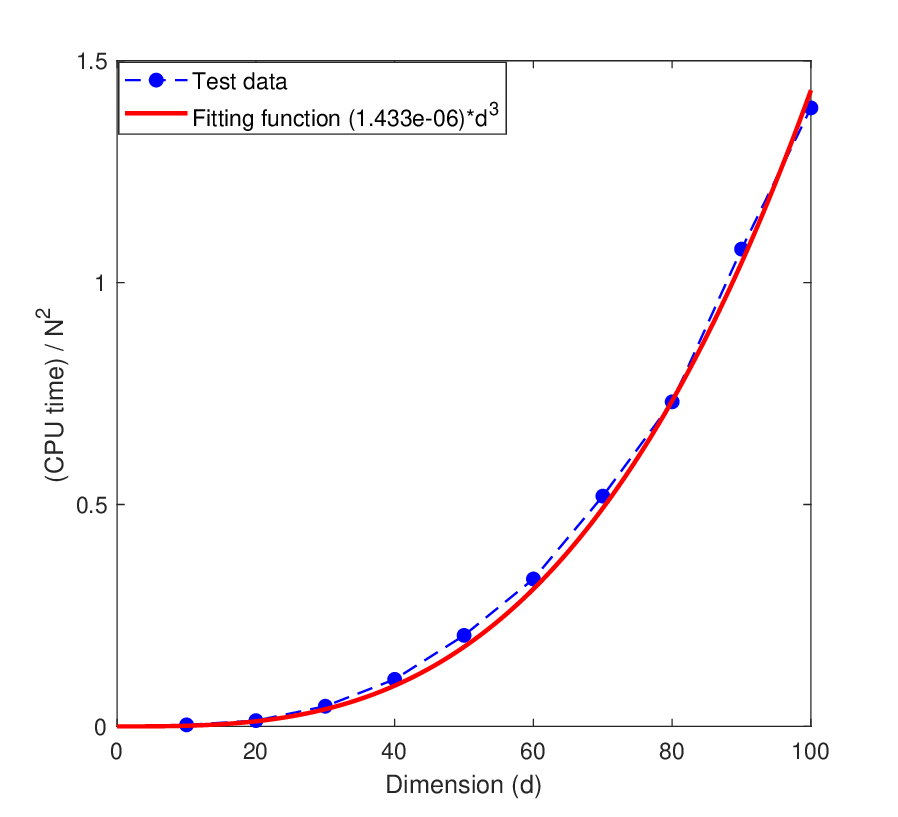}
	}
\end{figure}
\begin{figure}[H]
		\centerline{
		\includegraphics[width=1.75in,height=1.65in]{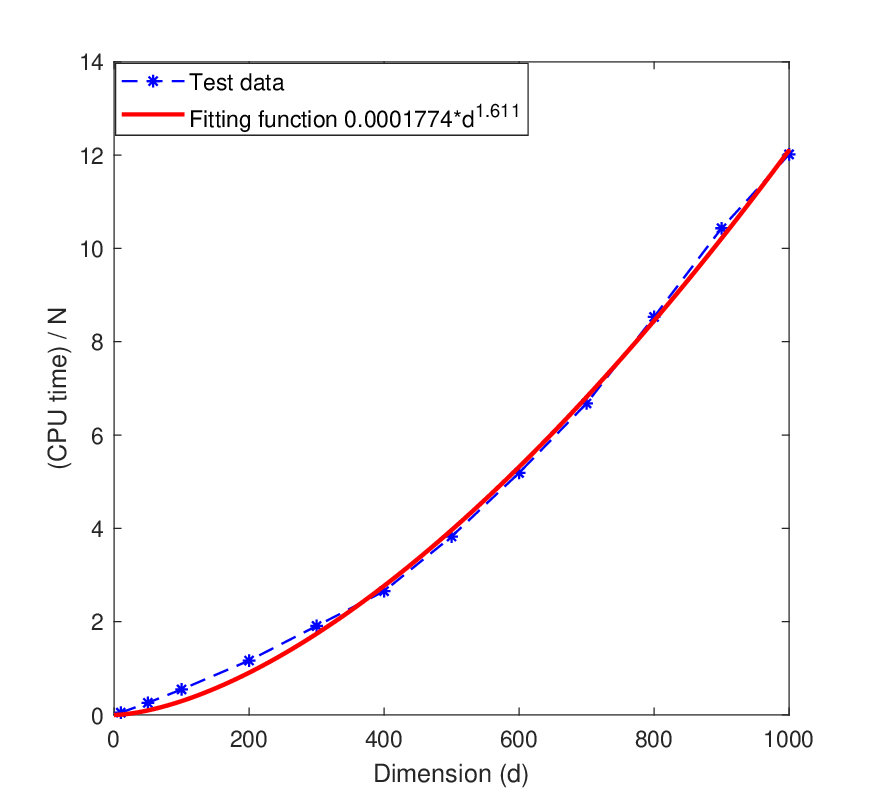}
		\includegraphics[width=1.75in,height=1.65in]{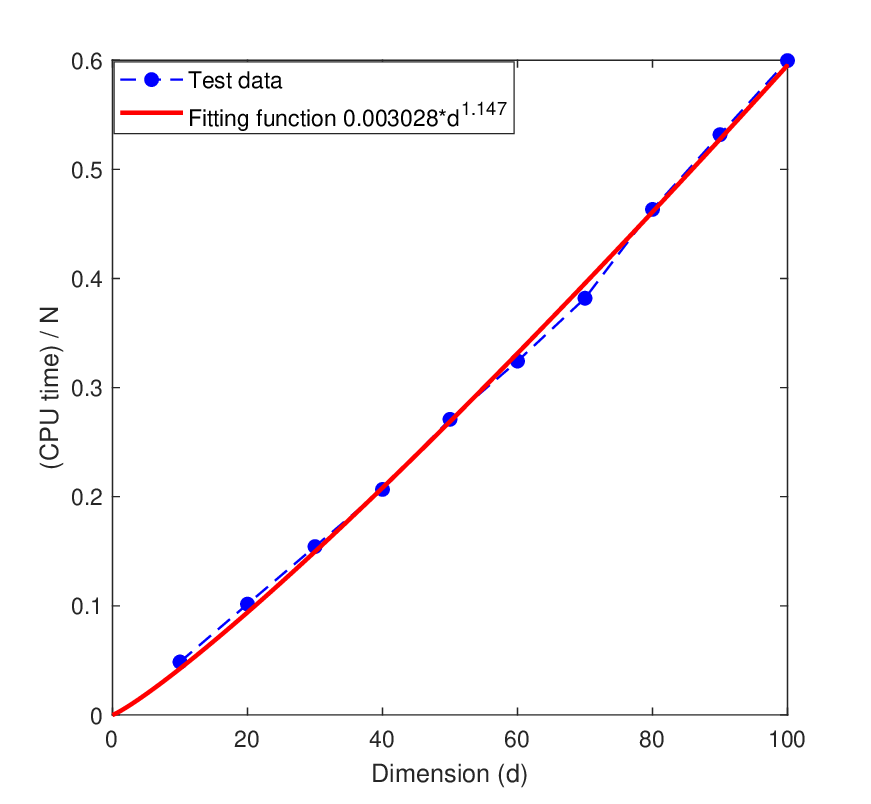}
		\includegraphics[width=1.75in,height=1.65in]{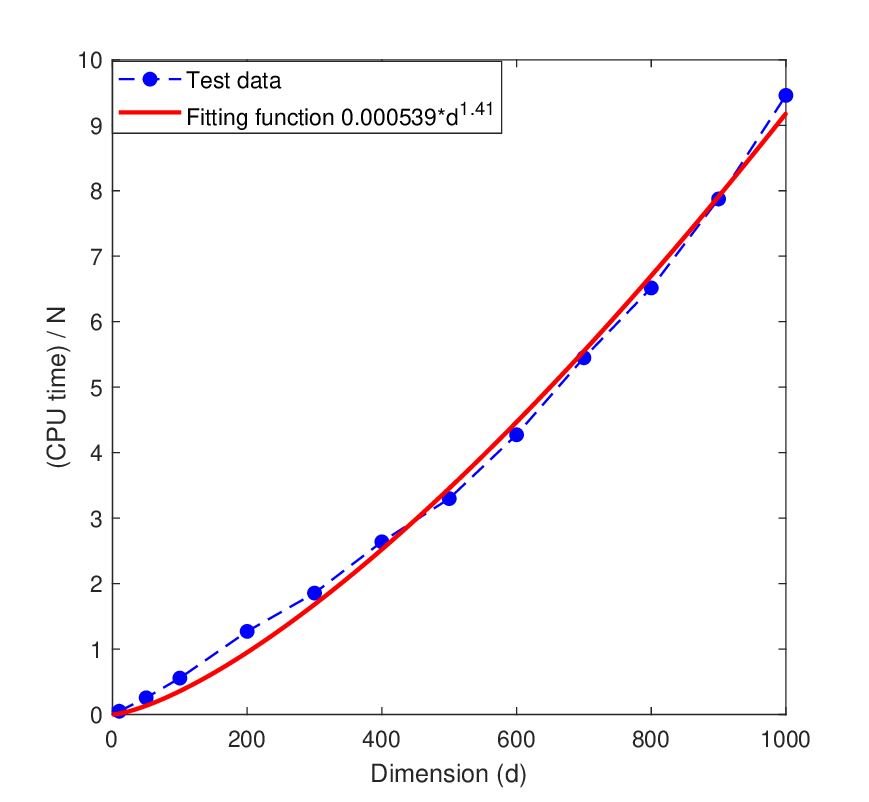}
	}	
	\centerline{
		\includegraphics[width=1.75in,height=1.65in]{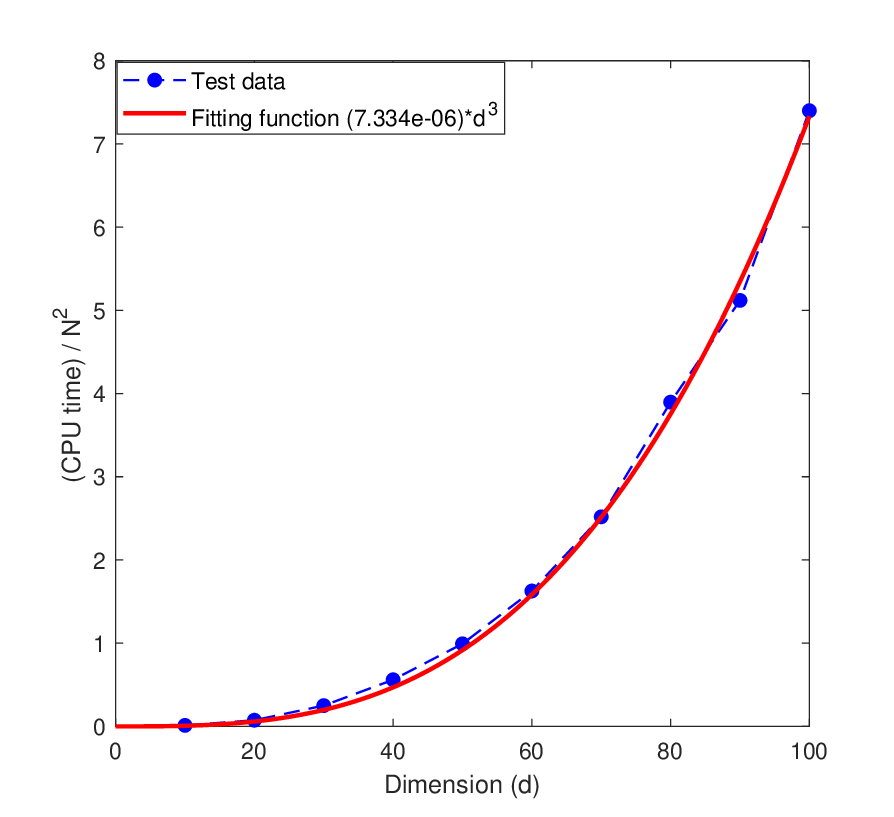}
		\includegraphics[width=1.75in,height=1.65in]{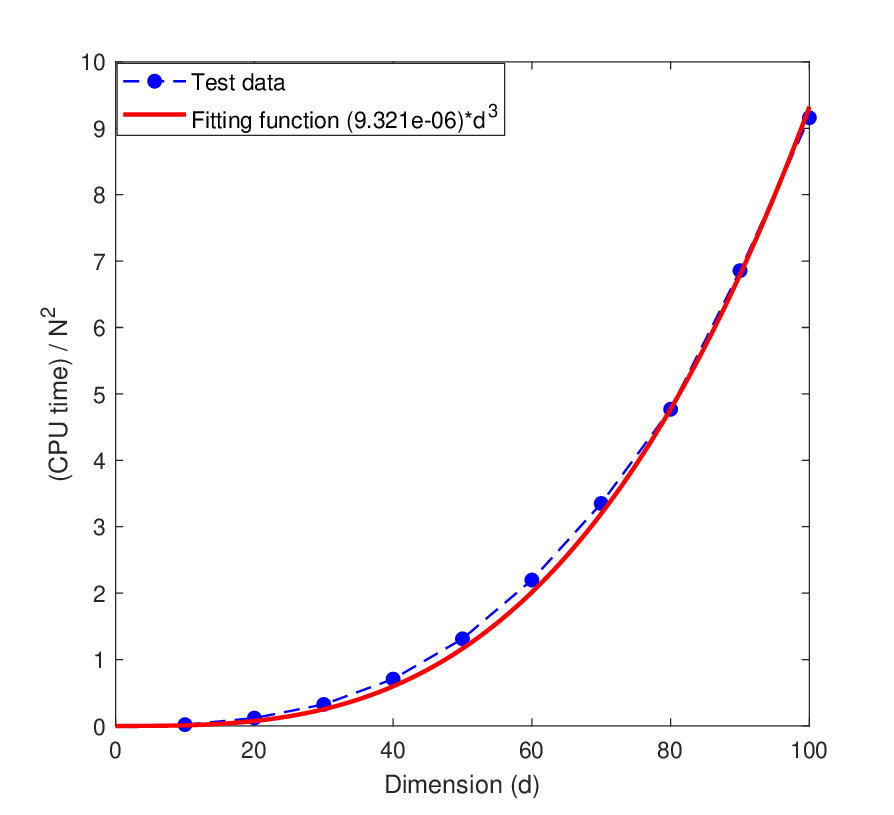}
		\includegraphics[width=1.75in,height=1.65in]{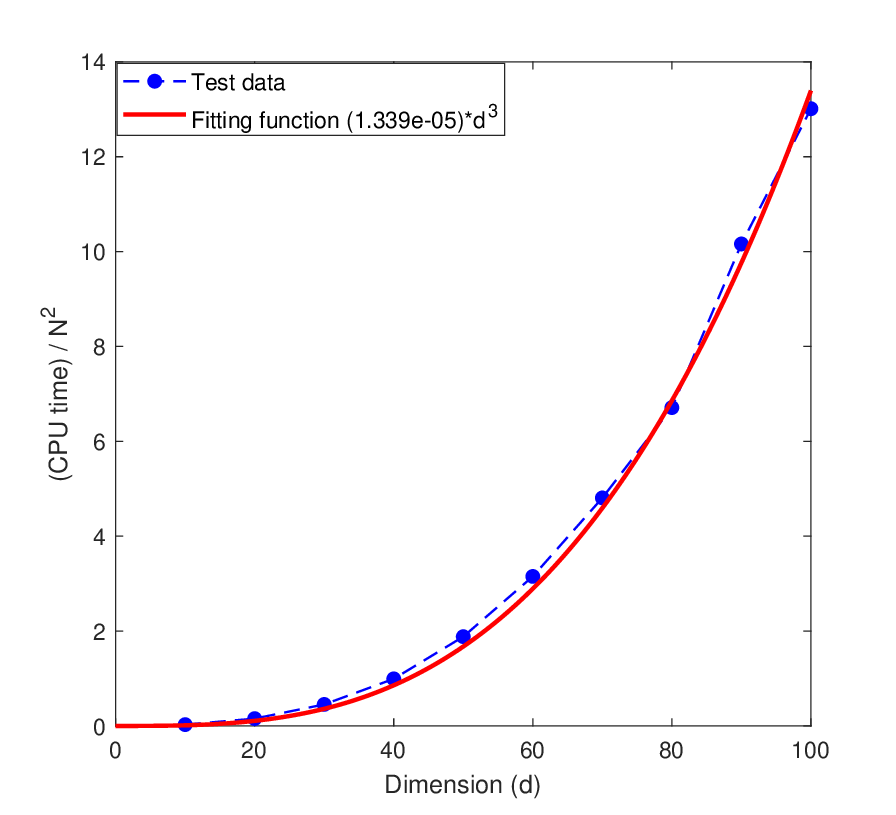}
	}
	\centerline{
		\includegraphics[width=1.75in,height=1.65in]{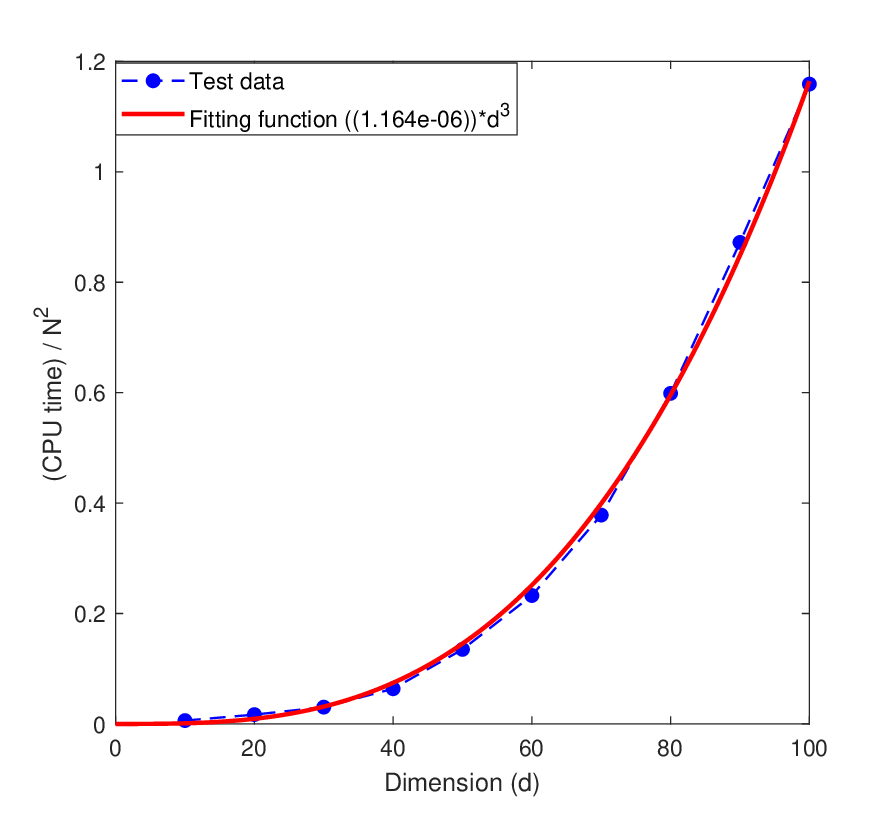}
		\includegraphics[width=1.75in,height=1.65in]{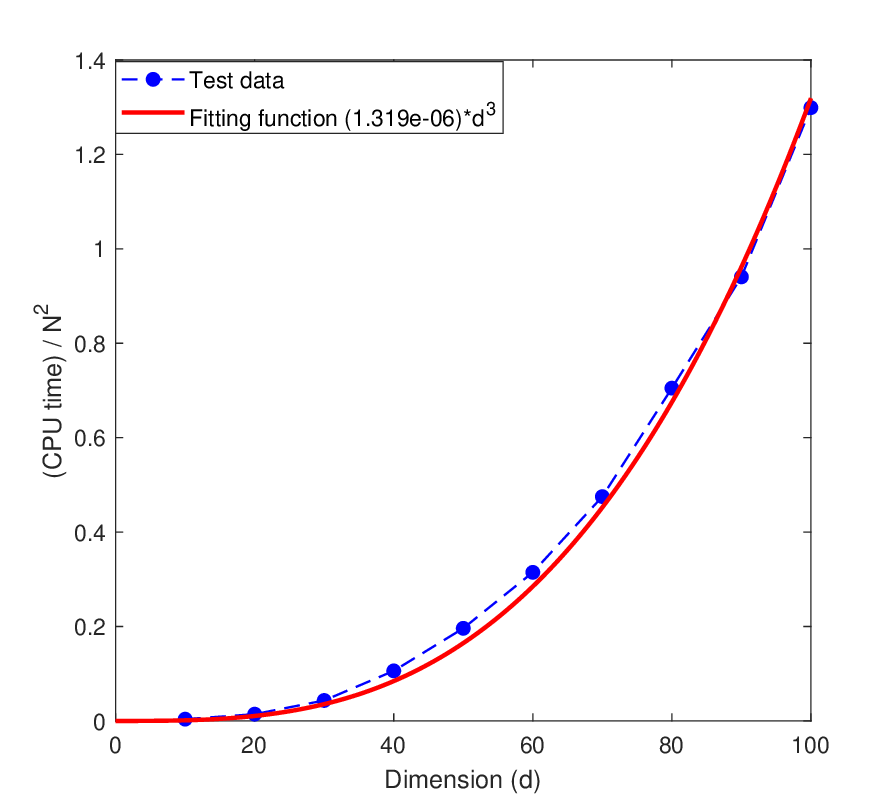}
		\includegraphics[width=1.75in,height=1.65in]{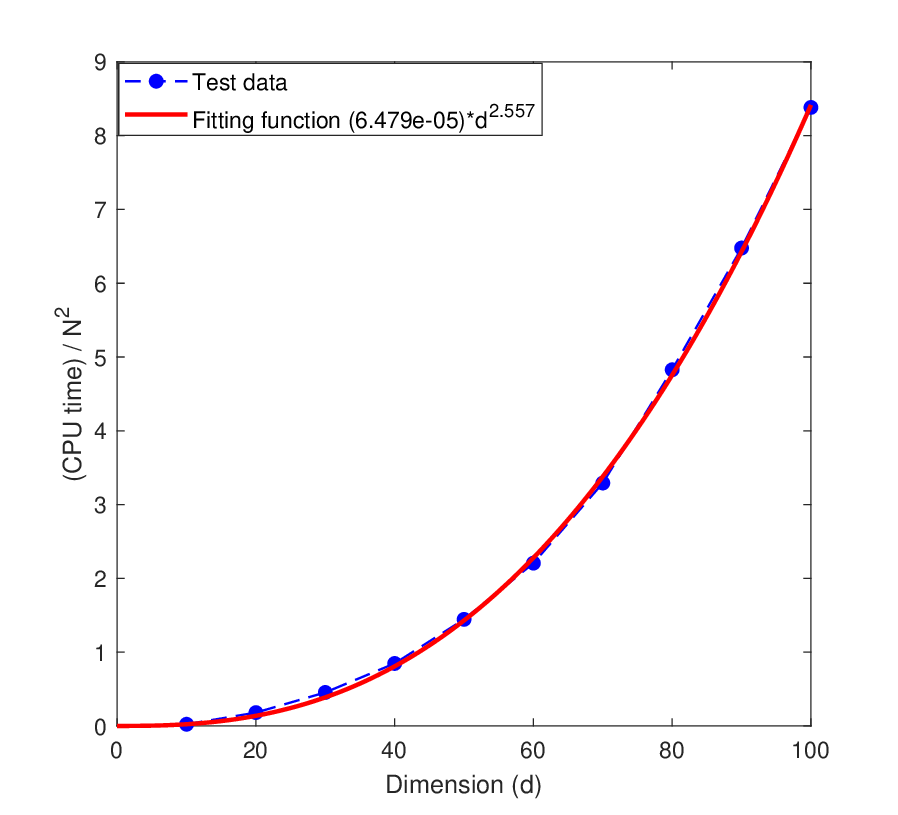}
	} 
	\centerline{
		\includegraphics[width=1.75in,height=1.65in]{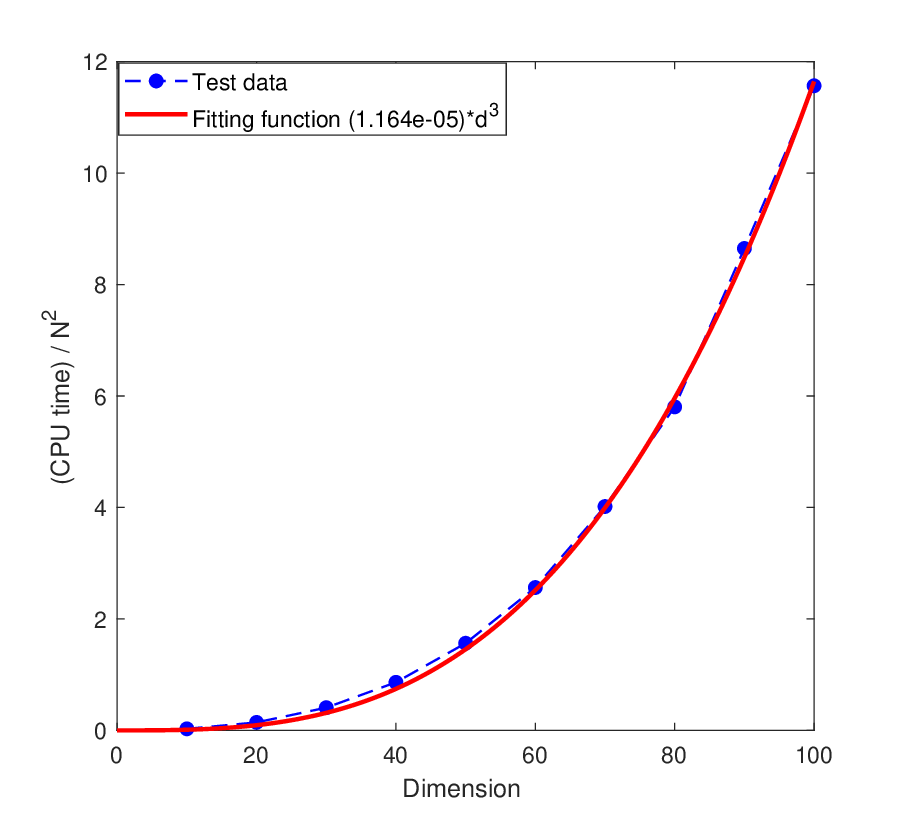}
		\includegraphics[width=1.75in,height=1.65in]{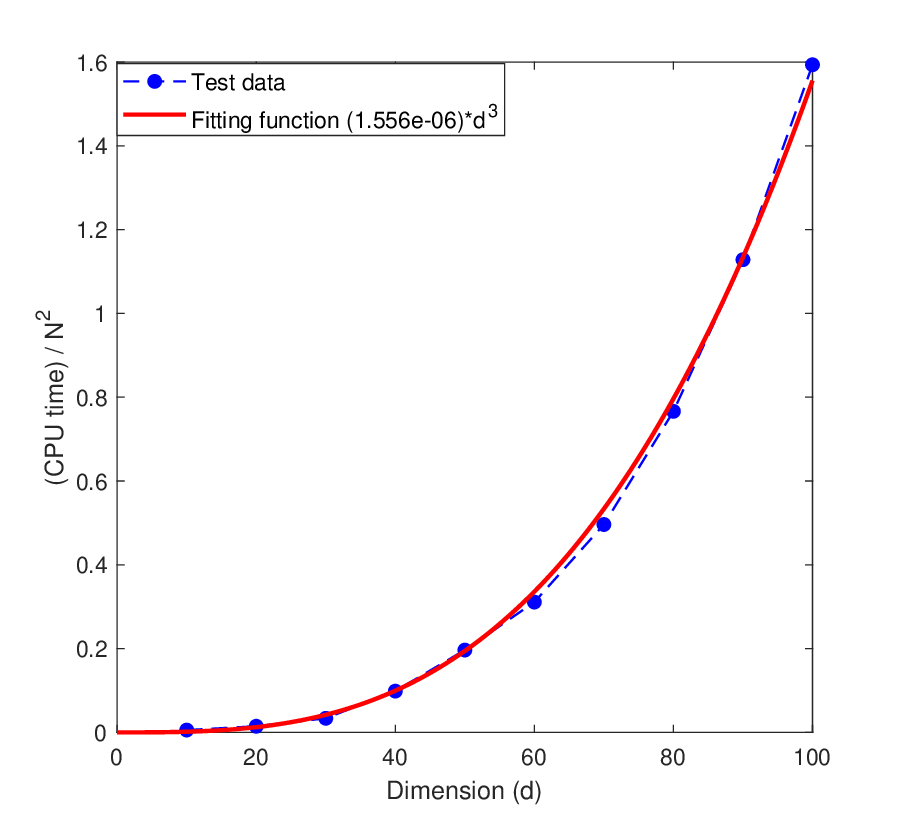}
		\includegraphics[width=1.75in,height=1.65in]{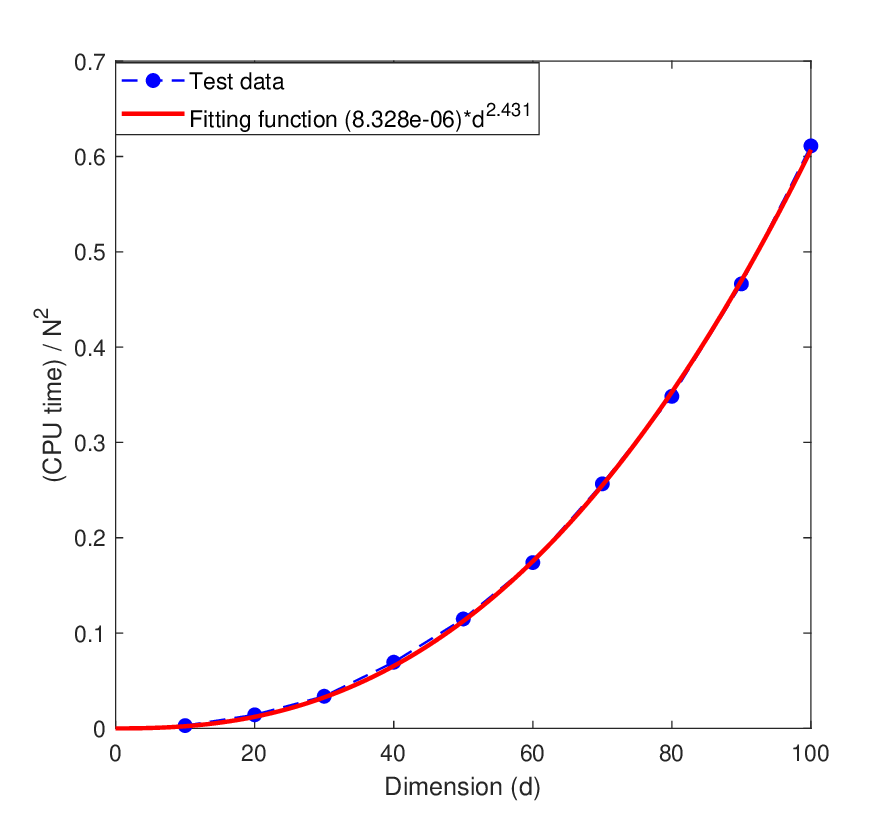}
	} 
	\caption{The relationship between the CPU time and dimension $d$.}	\label{fig:7} 
\end{figure}
\section{Conclusions}\label{sec-7}
In this paper, we introduced an efficient and fast algorithm MDI-LR  for implementing QMC lattice rules for high-dimensional numerical integration. It is based on the idea of converting and extending them into tensor product rules by affine transformations, and adopting the multilevel dimension iteration approach which 
computes  the function evaluations (at the integration points) in the multi-summation in cluster and iterates along each (transformed) coordinate direction so that a lot of computations can be reused. Based on numerical simulation results,  it was concluded that the computational complexity of the algorithm MDI-LR  (in terms of CPU time) grows at most cubically in the dimension $d$ and has an overall growth rate $O(d^3N^2)$, which suggests that the proposed algorithm MDI-LR  can effectively mitigate the curse of dimensionality in high-dimensional numerical integration, making the QMC lattice rule not only competitive but also practically useful for high dimension numerical integration.  Extensive numerical tests were provided to guage the performance of the algorithm MDI-LR  and to compare its performance with the standard QMC lattice rules.  Extensions to general Monte Carlo methods and applications of the proposed MDI-LR algorithm for solving  high-dimensional PDEs will be explored and reported in a forthcoming work.


\end{document}